\documentclass[10pt]{amsart}

\usepackage{enumerate,url,amssymb}
\usepackage{graphicx}
\usepackage{psfrag}

\newtheorem{theorem}{Theorem}[section]
\newtheorem{lemma}[theorem]{Lemma}
\newtheorem{proposition}[theorem]{Proposition}
\newtheorem{corollary}[theorem]{Corollary}

\theoremstyle{definition}
\newtheorem{definition}[theorem]{Definition}
\newtheorem{example}[theorem]{Example}

\newtheorem{question}[theorem]{Question}

\theoremstyle{remark}
\newtheorem{remark}[theorem]{Remark}

\numberwithin{equation}{section}

\newcommand{\R}{\mathbb{R}}

\newcommand{\dtext}{\textnormal d}
\DeclareMathOperator{\diam}{diam}
\DeclareMathOperator{\dist}{dist}

\DeclareMathOperator{\re}{Re}
\DeclareMathOperator{\im}{Im}

\DeclareMathOperator{\loc}{loc}

\def\XXint#1#2#3{{\setbox0=\hbox{$#1{#2#3}{\int}$}
\vcenter{\hbox{$#2#3$}}\kern-.5\wd0}}

\def\le{\leqslant}
\def\ge{\geqslant}

\begin{document}

\title{Dynamics of  Quasiconformal  Fields}

\author{Tadeusz Iwaniec}
\address{Department of Mathematics, Syracuse University, Syracuse,
NY 13244, USA}
\email{tiwaniec@syr.edu}
\thanks{Iwaniec was supported by the NSF grant DMS-0800416.}

\author{Leonid V. Kovalev}
\address{Department of Mathematics, Syracuse University, Syracuse,
NY 13244, USA}
\email{lvkovale@syr.edu}
\thanks{Kovalev was supported by the NSF grant DMS-0700549.}

\author{Jani Onninen}
\address{Department of Mathematics, Syracuse University, Syracuse,
NY 13244, USA}
\email{jkonnine@syr.edu}
\thanks{Onninen was supported by the NSF grant  DMS-0701059.}


\date{November 20, 2008}


\begin{abstract}
A uniqueness theorem is established for autonomous systems of ODEs, $\dot{x}=f(x)$, where $f$ is a Sobolev vector field with additional geometric structure, such as delta-monoticity or reduced quasiconformality. Specifically, through every non-critical  point of $f$ there passes a unique integral curve.
\end{abstract}

\maketitle
 \tableofcontents

\section{Introduction and Overview}
Let  $f \colon \Omega \to \R^n$ be a continuous vector field  defined in a domain $\Omega \subset \R^n$.  We shall consider    the associated autonomous system of ordinary differential equations with given initial data
\begin{equation}\label{IVP}
\dot{x}(t)=f\big(x(t)\big)\, , \quad \quad x(t_0)=x_0 \in \Omega
\end{equation}
By virtue of Peano's Existence Theorem, the system admits a local solution; that is, a solution defined in an open interval containing $t_0$, in which we have $x(t)\in \Omega$. However, uniqueness of the local solution is not always guaranteed. Every local solution $x=x(t)$ can be extended (as a solution in $\Omega$) to its  maximal interval of existence, say for $t\in (\alpha, \beta)$   where $-\infty \leqslant \alpha <  \beta \leqslant \infty$. Such  an interval will, of course, depend upon the choice of extension of the local solution. The  limits $\lim\limits_{t \searrow \alpha} x(t)$ and $\lim\limits_{t \nearrow \beta} x(t)$, if exist in $\Omega$, are the critical points of $f$; that is, zeros of $f$. The classical theory of ODEs tells us that Lipschitz vector fields admit unique local solutions; for less regular fields the solutions are seldom unique,  see   \cite[Ch. I Corollary 6.2]{Habook} for related results.   In the present paper, we address the uniqueness question  under significantly weaker regularity hypothesis on $f$.
We work with fields $f$ that are locally in Sobolev class $W^{1,p}$ for some $n<p<\infty$. The DiPerna-Lions theory (see~\cite{DL}, \cite{Am}
and references therein) establishes the existence and uniqueness of suitably generalized flow for Sobolev fields under certain restrictions
on their divergence. Our results are different in that we obtain the uniqueness of solutions in the classical sense, for all initial values
except for critical points. In order to achieve this, the geometry of $f$ (e.g., quasiconformality or monotonicity)
must come into play. It should be noted that the fruitful connection between the theory of ODEs and
geometric function theory has a long history \cite{Ah, BHS, Re, Sa, Se}.

It is easily seen that monotonicity of $f$ yields backwards uniqueness  \cite[Ch. III Theorem 6.2]{Habook}.
\begin{definition}
A continuous vector field $f\colon \Omega \to \R^n$ is said to be {\it monotone} if
\begin{equation}\label{mon}
\left\langle f(a)-f(b)\, , \, a-b \right\rangle \geqslant 0 \quad \textnormal{ for every } a,b \in \Omega
\end{equation}
It is {\it strictly monotone} if equality occurs only for $a=b$.
\end{definition}
\begin{proposition}$\textnormal{\textsc{(Backward Uniqueness)}}$
Suppose $f \colon \Omega \to \R^n$  is monotone and $x=x(t)$ and $y=y(t)$ are solutions to the system (\ref{IVP}) in $\Omega$. Then the distance between them,  $t \to |x(t)-y(t)|$, is nondecreasing. In particular, if $x(t_0)=y(t_0)$, then $x(t)=y(t)$ for all values   $t\leqslant t_0$ in the range of existence of $x(t)$ and $y(t)$.
\end{proposition}
We include a short proof of this proposition,  mainly to keep the exposition as self contained as possible.
\begin{proof}
We have
\begin{equation}
\frac{\dtext}{\dtext t} \left|x (t)-y(t)\right|^2 =2 \left\langle \dot{x}-\dot{y}\, , \, x-y \right\rangle  = 2 \left\langle f(x)-f(y)\, , \, x-y \right\rangle \geqslant 0 \nonumber
\end{equation}
Hence, for $t \leqslant t_0$ we obtain  $\left|x (t)-y(t)\right| \leqslant \left|x (t_0)-y(t_0)\right|=0$.
\end{proof}

In general,  forward uniqueness fails for monotone fields (Example \ref{Ex1}), although it holds for almost every initial value~\cite{Ce}.
However, we shall prove that forward uniqueness for $\delta$-monotone fields, holds for every noncritial initial value.
\begin{definition}
A vector field $f \colon \Omega \to \R^n$ is called {\it $\delta$-monotone}, $0< \delta \leqslant 1$, if for every $a,b \in \Omega$
\begin{equation}\label{demon}
\left\langle f(a)-f(b)\, , \, a-b \right\rangle \geqslant  \delta \left| a-b \right| \left|f(a)-f(b)\right|
\end{equation}
\end{definition}
Note that there is no supposition of continuity here. In fact,  a nonconstant $\delta$-monotone mapping is not only continuous but also   a $K$-quasiconformal homeomorphism \cite{K}, see Section \ref{SecQS} for the definition of $K$-quasiconformality.
\begin{theorem}\label{ThDelMon}   Let $f \colon \Omega \to \R^n$ be nonconstant $\delta$-monotone. Then the initial value problem
\begin{equation}
\dot{x}(t)=f\big(x(t)\big)\, , \quad \quad x(0)=x_0\in \Omega
\end{equation}
admits unique local solution, provided $f(x_0) \ne 0$.
\end{theorem}
The condition $f(x_0) \ne 0$ turns out to be necessary, though it is redundant for Lipschitz vector fields, see Example \ref{Ex1}.

It is also not difficult to see that if $f$ is merely H\"older continuous,  $f\in C^\alpha (\Omega)$,  with $0< \alpha <1$, then the assumption $f(x_0) \ne 0$ does not guarantee uniqueness. However, the uniqueness of integral curves, even for only  H\"older regular vector fields,  is still possible under additional  geometric  conditions, like $\delta$-monotonicity   in Theorem \ref{ThDelMon}. We shall work with homeomorphisms $f \colon   \R^n  \overset{\textnormal{\tiny{onto}}}{\longrightarrow} \R^n $ normalized by $f(0)=0$, so the origin of $\R^n$ is the only critical point of the field. In  the complex plane there is a close relationship between monotone vector fields and the so-called reduced quasiconformal mappings. In Section \ref{SubSec23} we take a close look at the  reduced distortion inequality
\begin{equation}\label{BeIn}
\left|f_{\bar z} \right| \leqslant k \left| \re f_z \right| \, , \quad 0 \leqslant k <1 \quad \textnormal{ for } f \in W^{1,1}_{\loc}(\mathbb C)
\end{equation}
This concept and relevant results  can be found in   \cite{AIMbook} and the recent work by the authors \cite{IKO}. One unusual feature of the homeomorphic solutions to the reduced distortion inequality should be pointed out. The measurable function $ \re f_z$ does not change sign in $\mathbb C$, see \cite[Theorem 6.3.2]{AIMbook}. Precisely, we have
\begin{equation}\label{RePo}
\textnormal{either } \hskip2cm \re f_z \geqslant 0 \quad \textnormal{ a.e.  in } \mathbb C
\end{equation}
\begin{equation}\label{ReNe}
\textnormal{or } \hskip2.7cm \re f_z \leqslant 0 \quad \textnormal{ a.e.  in } \mathbb C
\end{equation}
What is more, though we do not exploit it here, is  that (\ref{RePo}) or  (\ref{ReNe}) actually hold with strict inequalities, which is rather deep analytic fact recently established by  Alessandrini  and  Nesi  \cite{AN} in connection with the question of $G$-compactness of the Beltrami equation \cite{IGKMS, AIMbook}. The property (\ref{RePo})-(\ref{ReNe}), does not hold for noninjective solutions of (\ref{BeIn}). It also fails for homeomorphic solutions in proper subdomains $\Omega \subset \mathbb C$. Since we confine ourselves to injective vector fields in the entire complex plane, we can certainly assume that $\re f_z \geqslant 0$. Thus,
\begin{equation}\label{Kav18}
\left| f_{\bar z} \right| \leqslant k \re f_z
\end{equation}
For if not,   replace $f$ by $-f$, which  affects  only the direction of the integral curves.
Homeomorphic solution to~\eqref{Kav18} will be referred to as reduced $K$-quasiconformal mappings, $K= \frac{1+k}{1-k}$.
In Section \ref{Sec3} we shall show that
\begin{proposition}\label{Pr16}
Every nonconstant solution to the reduced distortion inequality
\begin{equation}
\left|f_{\bar z} \right| \leqslant k  \re f_z  \, , \quad 0 \leqslant k <1 \quad \textnormal{ for } f \in W^{1,1}_{\loc}(\mathbb C) ,
\quad f(0)=0\nonumber
\end{equation}
is strictly monotone, unless $f(z)=i \omega z$, where $\omega$ is a (nonzero) real number, in which case $\langle f(a)-f(b), a-b \rangle \equiv 0$.
\end{proposition}
We refer to this latter case as degenerate reduced quasiconformal map. Dynamics of the vector field $f(z)=i \omega z$ is rather simple; its integral curves are circles centered at the origin, $z(t)= z_0 e^{i \omega t}$. From now on let us  restrict ourselves to discussing nondegenerate reduced $K$-quasiconformal fields. Since $f$ is strictly monotone, it follows that $\re f(1)= \langle f(1)-f(0), 1-0 \rangle >0$. Thus it involves no loss of generality in assuming that $\re f(1)=1$, by rescaling   time parameter if necessary. This yields $|f(1)|\geqslant 1$. We then introduce the following class of the reduced $K$-quasiconformal fields.
\begin{definition}
Given $K \geqslant 1$ and $d \geqslant 1$, we consider the family
\[\mathcal F_K(d)= \big\{ f \colon \mathbb C \to \mathbb C  \colon f(0)=0 \textnormal{ and } 1 = \re f(1) \leqslant |f(1)| \leqslant d \big\}\]
where the mappings in consideration are solutions to the differential inequality
\begin{equation}
\left|f_{\bar z} \right| \leqslant \frac{K-1}{K+1}  \re f_z    \quad  \quad  f \in W^{1,2}_{\loc}(\mathbb C) \nonumber
\end{equation}
\end{definition}
Such solutions are automatically $K$-quasiconformal homeomorphisms. It is perhaps worth noting that $\mathcal F_K(d)$ is a convex family; that is, given $f,g\in \mathcal F_K(d)$ their convex combination
$(1- \lambda )f+ \lambda g$, $0 \leqslant \lambda \leqslant 1$, also belongs to  $\mathcal F_K(d)$.

 Let us summarize the above discussion by  the following chain of inclusions
\begin{eqnarray}\label{CoI} &\, & \,    \begin{Bmatrix} \delta \textnormal{-monotone}\\
 \textnormal{mappings}  \end{Bmatrix}  \subsetneq    \begin{Bmatrix}  \textnormal{reduced } K\textnormal{-quasiconformal} \\
\textnormal{mappings}  \end{Bmatrix}  \nonumber \\   &\, & \; \subsetneq   \begin{Bmatrix}  \textnormal{monotone } K\textnormal{-quasiconformal} \\
\textnormal{mappings}  \end{Bmatrix}  \subsetneq   \begin{Bmatrix}  K\textnormal{-quasiconformal} \\
\textnormal{mappings}  \end{Bmatrix}
\end{eqnarray}
Here  all the   inclusions are strict and $K= \frac{1+ \sqrt{1-\delta^2}}{1- \sqrt{1-\delta^2}}$.

We succeeded in extending Theorem~\ref{ThDelMon} to mappings in the second family of this chain.
\begin{theorem}\label{ThMain}$\textnormal{\textsc{(Uniqueness)}}$ Given a reduced quasiconformal field $f\in \mathcal F_K(d)$. Through every $x_0 \ne 0$ there passes exactly one integral curve $x=x(t)$,
\begin{equation}
\dot{x}(t)=f\big(x(t)\big)\, , \quad \quad x(0)=x_0
\end{equation}
defined in its maximal interval $(\alpha, \beta)$,   where  $-\infty \leqslant \alpha <0 < \beta \leqslant \infty$. We have
$x(t)\in \mathbb C_0=\mathbb C\setminus\{0\}$ for $t\in (\alpha, \beta)$, and
\begin{equation}\label{EqI111}
 \lim\limits_{t \searrow \alpha} x(t) =0 \quad \mbox{ and }  \quad \lim\limits_{t \nearrow \beta} x(t) =\infty
 \end{equation}
 Moreover,
 \begin{equation}\label{EqI112}
 |x(s)|<|x(t)|\, , \quad \quad \mbox{ for } \alpha <s<t<\beta
 \end{equation}
 \end{theorem}
In other words, as the point $x(t)$ travels along  the curve in the direction of the increasing parameter $t$, its distance to the critical point of $f$ strictly increases.  To accommodate explicit uniform estimates we    restrict     the integral curves to the  annulus
\[ \mathbb C_r^R = \left\{z \colon r \le |z| \le R  \right\} \, , \quad \quad 0 <r<R< \infty \]
Consider two integral curves
\begin{eqnarray*}
\dot{x}(t)= f \big(x(t) \big)\, , & \quad x(t_0) = x_0 \in \mathbb C_r^R \\
\dot{y}(s)= f \big(y(s) \big)\, , & \quad y(t_0) = y_0 \in \mathbb C_r^R,
\end{eqnarray*}
where the time parameters $t$ and $s$ are restricted to the intervals in which
\[r   \leqslant \left| x(t )\right| \le R \quad \textnormal{and} \quad  r    \leqslant \left| y(s)\right| \le R \]
respectively. We then have the following Lipschitz dependence of the solutions on both the time parameter and initial data.

\begin{theorem}\label{ThMain2}$\textnormal{\textsc{(Lipschitz Continuity)}}$
There exist constants $A_r^R = A_r^R(K,d)$ and $B_r^R = B_r^R(K,d)$ such that
\begin{equation}\label{EqI1122}
\left| x(t)-y(s) \right| \leqslant A_r^R \left| t-s  \right| + B_r^R \left| x_0 -y_0 \right|
\end{equation}
In particular,
\begin{equation}\label{EqI113}
\left| x(t)-y(t) \right| \leqslant   B_r^R \left| x(s)-y(s) \right|
\end{equation}
as long as $x(t), y(t), x(s)$ and $y(s)$ lie in the annulus $\mathbb C_r^R$.
\end{theorem}

There is a convenient and geometrically pleasing parametrization of the integral curves for a given field  $f\in \mathcal F_K(d)$.  Every integral curve $\Gamma$ intersects the unit circle $\mathbb S^1$ at exactly one point $e^{i \theta}$, $0 \leqslant \theta < 2 \pi$.  Denote such curve by $\Gamma_\theta$ and call $\theta$ a quasipolar  angle of the curve.  We have already mentioned that if  a point $z$ moves along   $\Gamma_\theta $   its distance to the origin strictly increases in time. Thus $\rho = |x|$ can be used as a new parameter in $\Gamma$, $0 \leqslant \rho \leqslant \infty$.  In this way every point $z\in \mathbb C_0$ is uniquely prescribed by its quasipolar coordinates  associated with the vector field $f\in \mathcal F_K(d)$ . This is a pair $(\rho, e^{i\theta}) \in \R_+ \times \mathbb S^1$ with $\rho=|z|$ as the polar distance of $z$ and $\theta $ as its  quasipolar angle; for example,  the identity map $f=id \colon \mathbb C \to \mathbb C$ gives the usual   polar coordinates $(\rho, e^{i \theta})$ of  $z= \rho e^{i \theta}$. Quasipolar coordinates give rise to a homeomorphism  $\Phi: \mathbb C \to \mathbb C$ defined by the rule $\Phi (\rho, e^{i \theta}) = \rho \cdot e^{i \theta}$. This homeomorphism  turns out to be locally bi-Lipschitz in $\mathbb C_0$. Precisely, we have
\begin{equation}\label{EqI15}
c_r^R \le \left|   \frac{\Phi (z_1) - \Phi (z_2)  }{z_1-z_2} \right| \le C_r^R \quad \quad \mbox{for } z_1, z_2 \in \mathbb C_r^R
\end{equation}
see Theorem \ref{Th141}. Moreover, $\Phi$ is the identity  on $\mathbb S^1$ and takes every circle $\mathbb S_\rho =\{z \colon |z|=\rho  \}$ onto itself. More importantly, $\Phi$ rectifies each trajectory $\Gamma_\theta$ by mapping it  onto the straight ray
\[\mathcal R_\theta = \big\{ z \colon \arg z = \theta  \big\}\, , \quad 0 \leqslant \theta < 2 \pi\]
see Figure \ref{fig1}.

\begin{figure}[!h]
\begin{center}
\includegraphics*[height=2.5in, angle=90]{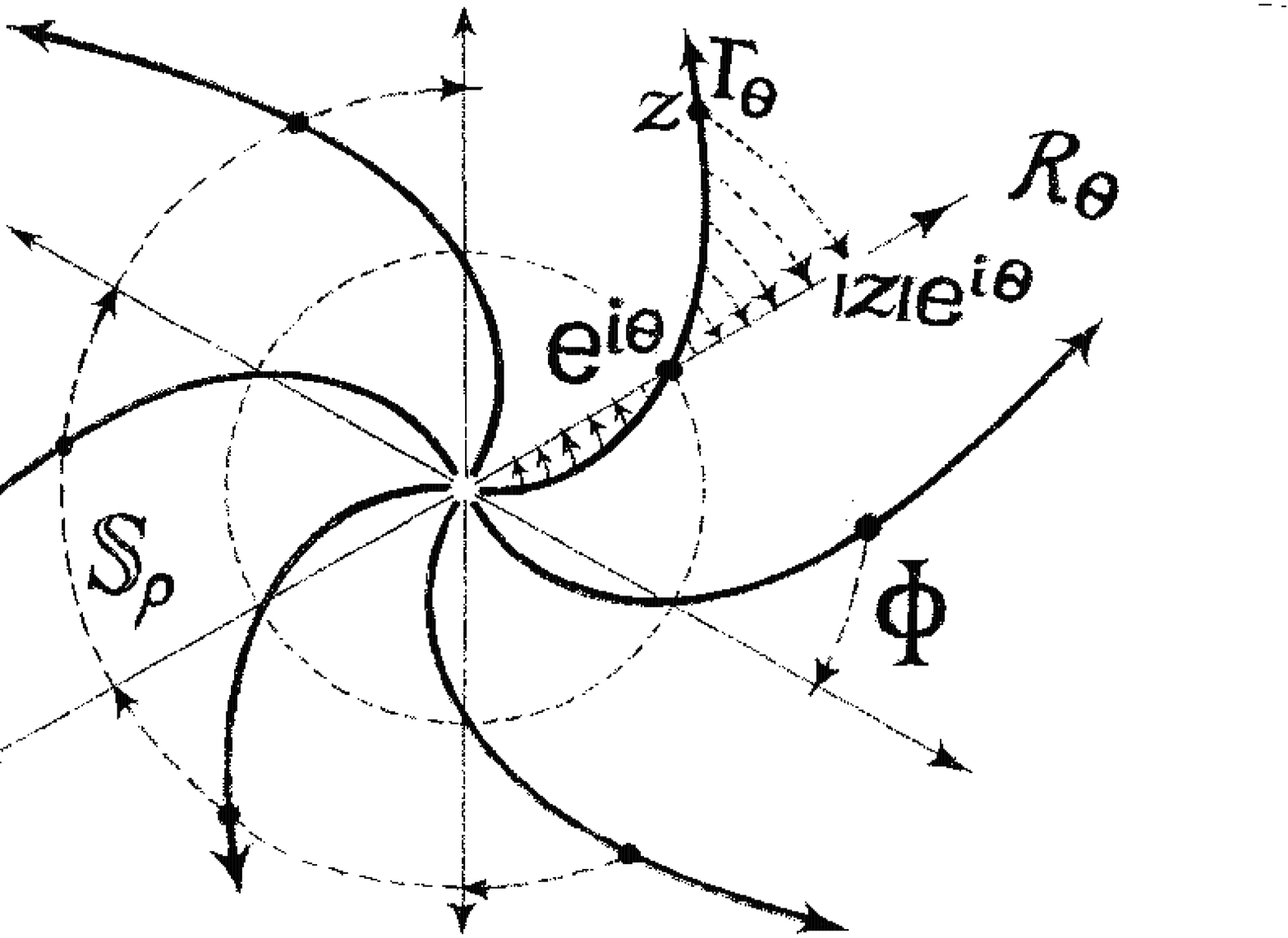}
\caption{Bi-Lipschitz rectification of trajectories.}\label{fig1}
\end{center}
\end{figure}

Every complex number $z \ne 0$ has infinitely many quasipolar angles which differ from each other by multiple of $2 \pi$.  Let us denote by $\Theta = \Theta (z)= \Theta_f(z)$ the multivalued function that assigns to $z\in \mathbb C_0$ all its quasipolar angles. A monodromy principle tells us that $\Theta$ has a continuous branch on every simply connected domain $\Omega \subset \mathbb C_0$. Two such branches differ by a constant. Therefore, it makes sense to speak of the gradient of $\Theta$,  defined by
\begin{equation}
\nabla \Theta   = \left( \frac{\partial \Theta }{\partial x}\, , \, \frac{\partial \Theta }{\partial y}\ \right) \, , \quad \mbox{ almost everywhere in } \mathbb C
\end{equation}
With the aid of the function $\Theta = \Theta (z)$ we shall factor the vector field  $i  f(z)$  into a product of a gradient field and a scalar function.
\begin{theorem}\label{ThMain3}
The orthogonal vector field
$ V(x,y)= if(x+iy)$, $f \in \mathcal F_k(d)$,
admits a factorization
\begin{equation}\label{EqFac}
V(x,y) = \lambda(x,y) \nabla \Theta
\end{equation}
The integrating factor is bounded from below and from above,
\begin{equation}\label{EqI118}
m(|z|)\leqslant \lambda (z) \leqslant M(|z|), \quad \lambda (z)= \frac{|f(z)|}{|\nabla \Theta (z)|}
\end{equation}
Here the lower and upper bounds $m,M  \colon (0,\infty) \to \mathbb R_+$ are continuous functions. These functions depend only on the parameters $K$ and $d$ of  the family $\mathcal F_K(d)$.
\end{theorem}
It is not difficult to construct a vector field $f \in \mathcal F_K(d)$ for which no factorization of the form (\ref{EqFac}) together with (\ref{EqI118})  allows the integrating factor $\lambda$ to be continuous, see Example \ref{Ex111}  for details. Curiously, the curvature (in somewhat generalized sense) of the orthogonal trajectories is nonnegative, see Remark \ref{rem112} for an explanation.

To conclude the introduction, let us mention some of the ingredients of our proofs.
\begin{theorem}\label{Th41}
Let $f \colon \Omega \to \R^n$ be nonconstant  $\delta$-monotone. Then the image $f(\Gamma)$ of any $C^1$-curve $\Gamma \subset \Omega$ is locally rectifiable.
\end{theorem}
Due to this property Theorem \ref{ThDelMon} is a corollary of the following more general result.
\begin{theorem}\label{ThMain4}
Let $f  \colon \Omega \to \R^n$ be a $K$-quasiconformal field of bounded variation on $C^1$-curves. Suppose we are given two local solutions of the Cauchy problem
\begin{eqnarray*}
\dot{x}(t)= f \big(x(t)\big)\, , & \quad x(0)=a\in \Omega \\
\dot{y}(t)= f \big(y(t)\big)\, , & \quad y(0)=a\in \Omega
\end{eqnarray*}
where $f(a) \ne 0$. Then there exist $\epsilon >0$ such that $x(t)=y(t)$ for $-\epsilon < t < \epsilon$.
\end{theorem}

The conclusion of Theorem~\ref{Th41} fails for general quasiconformal maps~\cite{Tu}, even for reduced ones~\cite{IKO2}.
This is where the elementary but very useful  concept of the modulus of monotonicity
\begin{equation}
\Delta_f(a,b) = \left\langle f(a)-f(b), \frac{a-b}{|a-b|}  \right\rangle
\end{equation}
comes into play. We show that for reduced quasiconformal maps $\Delta_f$ has the same quasisymmetric behavior as the modulus of continuity for general quasiconformal maps. We exploit this property  by computing $\Delta_f$ at  suitably chosen  points on integral curves. Due to  a   cancellation property of  the modulus of monotonicity the sum of $\Delta_f$ over such partition points     is controlled by the quadratic variation of $f$.  On the other hand, for any planar quasiconformal map $f$ the quadratic variation over a $C^1$-arc is controlled by the diameter of its image. From this we deduce the uniqueness of solutions.

Our results raise the following
\begin{question}
Let  $f\colon \Omega \to \R^n$ be quasiconformal and $f(x_0) \ne 0$.  Does the system (\ref{IVP}) admit a unique local solution?
\end{question}

\section{Background}
Let us introduce the notation and briefly review basic concepts.
\subsection{Quasisymmetry}\label{SecQS} Let $\Omega$ be a domain in $\R^n$. A sense preserving homeomorphism $f \colon \Omega \to \R^n$ is said to be {\it $K$-quasiconformal}, $1 \leqslant K < \infty$, if
\begin{equation}\label{distine1}
\limsup\limits_{\epsilon \to 0} \frac{\underset{{|x-a|=\epsilon}}{\max} \left|f(x) - f(a)\right|}{\underset{{|y-a|=\epsilon}}{\min} \left|f(y) - f(a)\right|} \leqslant K \, , \quad \textnormal{ for every } a\in \Omega
\end{equation}
 It is well known that such mappings belong to the Sobolev class $W^{1,n}_{\loc}(\Omega , \mathbb R^n)$ and are   H\"older continuous with exponent $\alpha = \frac{1}{K}$. An analytic description of (\ref{distine1}) can be formulated (equivalently) via the so-called distortion inequality
\begin{equation}\label{Kav22}
\left| Df(x) \right|^n \leqslant \mathcal K \cdot J(x,f) \quad \quad \mbox{a.e.} \quad f \in W^{1,n}_{\loc}(\Omega, \R^n)
\end{equation}
for some $1 \leqslant \mathcal K < \infty$. Here $|Df(x)|$ stands for the norm of the differential matrix and $J(x,f)$ for the Jacobian determinant. In the complex plane  it reads as
\begin{equation}\label{Kav23}
\left|f_{\bar z} (z)\right| \leqslant k \left|f_{ z} (z)\right|\, , \quad k= \frac{K-1}{K+1} \quad \textnormal{a.e.}
\end{equation}
The $W^{1,n}$-solutions to the distortion inequality (\ref{Kav22}) or (\ref{Kav23}) (not necessarily injective) are referred to as $ K$-quasiregular mappings. Quasiregular mappings are also locally H\"older continuous. A purely geometric description of quasiconformality, which proves useful in our study here,  is  the   concept of quasisymmetry, also called three point condition.
\begin{theorem}$\textnormal{\textsc{(Three Points Condition)}}$  Let $f \colon \R^n \to \R^n$ be $K$-quasiconformal. Then
\begin{equation}\label{Kav24}
m_K \left(\frac{|x-a|}{|y-a|}\right) \leqslant \frac{|f(x)-f(a)|}{|f(y)-f(a)|} \leqslant M_K \left(\frac{|x-a|}{|y-a|}\right)
\end{equation}
for $a,x,y \in \R^n$, $a \ne y$, where
\[M_K(t)= C_K \max \left(t^K, t^\frac{1}{K}\right) \, , \quad \quad 0 \leqslant t < \infty\]
and
\[m_K(t)= \left[ M_K \left( t^{-1}\right) \right]^{-1}= C^{-1}_K \min \left(t^K, t^\frac{1}{K}\right)   \]
If $f$ keeps the origin fixed, $f(0)=0$, then
\begin{equation}
m_K \left(|x| \right) \, |f(1)| \leqslant \left| f(x) \right| \leqslant M_K \left( |x|\right) \, |f(1)|
\end{equation}
More generally,
\begin{equation}
m_K \left(\frac{|x|}{|y|}\right) \leqslant \frac{|f(x)|}{|f(y)|} \leqslant M_K \left(\frac{|x|}{|y|}\right)
\end{equation}
\end{theorem}
\subsection{Modulus of monotonicity} Let $f \colon \Omega \to \R^n$ be continuous and monotone. The modulus of monotonicity $\Delta_f \colon \Omega \times \Omega \to [0, \infty)$ is defined by the rule
\begin{equation}
0 \leqslant \Delta_f(a,b)= \begin{cases}  \left\langle f(a)-f(b), \frac{a-b}{|a-b|} \right\rangle & \quad \textnormal{if } a \ne b\\
0 & \quad \textnormal{if } a = b  \end{cases} \; \;  \leqslant \;  \left|f(a)-f(b) \right| \nonumber
\end{equation}
We shall also work with    $\delta_f \colon \Omega \times \Omega \to [0, 1]$  given by
\begin{equation}
\delta_f(a,b) =  \left\langle \frac{f(a)-f(b)}{|f(a)-f(b)|}   , \frac{a-b}{|a-b|} \right\rangle \quad \quad  \mbox{ for } a \ne b
\end{equation}
Thus  $f$ is  $\delta$-monotone if and only if
\begin{equation}
\delta_f (a,b) \geqslant \delta \, , \quad \quad 0 \le \delta <1
\end{equation}
or, equivalently,
\begin{equation}
\Delta_f(a,b) \geqslant \delta \left| f(a)-f(b) \right|\, , \quad \textnormal{ for all } a,b \in \Omega
\end{equation}
\subsection{Reduced $K$-quasiconformal fields}\label{SubSec23}
We will be dealing with     the reduced distortion inequality
\begin{equation}\label{Eq101}
\left| f_{\bar z} \right| \leqslant k \re f_z \, , \quad 0 \leqslant k = \frac{K-1}{K+1}<1\, , \quad f(0)=0
\end{equation}
for $f: \mathbb C \to \mathbb C$ in the  Sobolev space $W^{1,2}_{\loc} (\mathbb C)$.  Such solutions form  a convex cone in $W^{1,2}_{\loc} (\mathbb C)$.  Precisely, if $f_1$ and $f_2$ solve (\ref{Eq101})  and $\lambda_1, \lambda_2 \geqslant 0$, then so does the mapping $\lambda_1 f_1 + \lambda_2 f_2$. As an example, consider the linear map $f_\circ (z)= az + b \bar z$ in which $|b| \leqslant k \re a$. Adding $f_\circ$ to a solution   of (\ref{Eq101}) gives another solution $F(z)=f(z)+az + b\bar z$. We recall  rather unexpected topological fact  that every  nonconstant quasiregular mapping   $f \colon \mathbb C \to \mathbb C$,  with $\re f_z \geqslant 0$ almost everywhere, is a homeomorphisms of $\mathbb C$ onto $\mathbb C$ \cite{IKO}. Actually such mapping satisfies   strict inequality $\re f_z >0$, almost everywhere, except for the case of the  degenerate monotone mapping
\begin{equation}\label{Eq102}
f(z) \equiv i \omega\,   z \, , \quad \textnormal{ with some } \omega \in \mathbb R\setminus \{0\}
\end{equation}
 The integral curves    $\dot{z}= i \omega\,  z$   are circles centered at the origin, $z(t)= z_0 e^{i \omega\,  t}$. As this case is completely clear, we  shall focus on  nondegenerate  reduced $K$-quasiconformal fields; that is   mappings $f \colon \mathbb C \to \mathbb C$ of Sobolev class  $W^{1,2}_{\loc} (\mathbb C)$ such that
\begin{equation}\label{Eq103}
\left|  f_{\bar z} \right| \leqslant k \re f_z \, , \quad f(0)=0
\, , \quad f(z) \not\equiv i \omega z \, , \quad 0 \leqslant k <1
 \end{equation}
The simple case is the complex linear vector field  $f(z)= (\lambda +i\omega)z$, $\lambda >0$. Its trajectories  are spirals $z(t)=z_0 e^{\lambda t} e^{i \omega t}$, $- \infty < t < \infty$, except for $\omega =0$. In this latter case
\begin{equation}\label{Eq104}
f(z)= \lambda z \, , \quad \quad \lambda >0
\end{equation}
for which the trajectories are straight rays $z(t)= z_0\,  e^{\lambda t}$, $-\infty < t < \infty$.   We shall see latter  that the trajectories of  every   (nondegenerate) reduced $K$-quasiconformal field  are  images of straight rays via a homeomorphism  $\Psi \colon \mathbb C   \overset{\textnormal{\tiny{onto}}}{\longrightarrow} \mathbb C$, $\Psi(0)=0$. This homeomorphism turns out to be locally bi-Lipschitz on $\mathbb C_0$,
see section~\ref{BLsection}.

\section{Estimates of  reduced $K$-quasiconformal fields}\label{Sec3}
Let $\lambda + i \omega$ be a complex number in the right half plane, $\lambda \geqslant 0$. Given any (nondegenerate) reduced $K$-quasiconformal map $f \colon \mathbb C \to \mathbb C$, we consider a  map $F(z)=f(z) + (\lambda + i \omega)z$. This   is a nonconstant solution to the same distortion inequality as $f$. Indeed,
\[\left| F_{\bar z}   \right| = \left| f_{\bar z}   \right|  \leqslant k \re f_z \leqslant k \re F_z  \]
Thus $F$ is $K$-quasiregular.  By virtue of Corollary~1.5 \cite{IKO} $F$ is a  homeomorphism.  In particular, the three point condition applies to $F$ to yield the inequalities
\begin{equation}\label{Eq111}
m_K \left( \left| \frac{x-a}{y-a}\right| \right) \leqslant \frac{|x-a|  \left| \frac{f(x)-f(a)}{x-a} + \lambda + i \omega \right|}{ |y-a|  \left| \frac{f(y)-f(a)}{y-a} + \lambda + i \omega \right| } \leqslant M_K \left( \left| \frac{x-a}{y-a}\right| \right)
\end{equation}
for every $x$ and $y \ne a$. Therefore, $\frac{f(x)-f(a)}{x-a} + \lambda + i \omega \ne 0$, whenever $x \ne a$. Putting $\omega = - \im \frac{f(x)- f(a)}{x-a}$  we arrive at the inequality
\[\lambda + \re \frac{f(x)-f(a)}{x-a} \ne 0 \quad \quad \textnormal{for all } \lambda \geqslant 0 \quad \mbox{and} \quad x \ne a\]
This gives
\begin{equation}\label{Eq112}
\Delta_f (x,a) =   |x-a| \re \frac{f(x)-f(a)}{x-a}>0
\end{equation}
We just proved   that every (nondegenerate) reduced quasiconformal map is strictly monotone, as stated in Proposition \ref{Pr16}.

Setting $a=0$ and $x=1$, we obtain
\begin{equation}\label{Eq113}
 \re f(1)= \Delta_f (1,0) >0
\end{equation}
This inequality makes it legitimate to normalize the (nondegenerate) reduced $K$-quasiconformal fields  by the condition $\re f(1)=1$. We indeed have used such normalization    in the definition of the class $\mathcal F_K(d)$.  Let us record this normalization once again as
\begin{equation}\label{Eq33p}
\Delta_f(1,0) = \re f(1) =1
\end{equation}

We now return to (\ref{Eq111}), and set  $\lambda =0 $ and  $\omega = - \im \frac{f(x)- f(a)}{x-a}$. We then arrive at  the   same three point condition for  $\Delta_f$ as for a general $K$-quasiconformal mapping in (\ref{Kav24}),
\begin{equation}\label{Eq114}
\frac{\Delta_f (x,a)}{\Delta_f (y,a)} = \frac{ \left\langle f(x)-f(a), \frac{x-a}{|x-a|}\right\rangle}{\left\langle f(y)-f(a), \frac{y-a}{|y-a|}\right\rangle} \geqslant    m_K \left( \left| \frac{x-a}{y-a}\right| \right) =  M^{-1}_K \left( \left| \frac{y-a}{x-a}  \right|  \right)
\end{equation}
In particular, setting $a=0$, we obtain
\begin{equation}\label{Eq115}
m_K \left( \left| \frac{x}{y}\right| \right) \leqslant \frac{\Delta_f (x,0) }{\Delta_f (y,0)} \leqslant M_K \left( \left| \frac{x}{y}\right| \right)
\end{equation}
Then, letting $y=1$, this simplifies to:
\begin{equation}\label{Eq116}
m_K(|x|)  \leqslant  |x| \re \frac{f(x)}{x}=\Delta_f (x,0)  \leqslant   M_K \big(|x|\big)\, , \quad \mbox{ for } f \in \mathcal F_K(d)
\end{equation}
As for the estimate of $|f(x)|$, we may  use the three point condition  and the assumption that $1\leqslant |f(1)|\leqslant d$, to infer that
\begin{equation}\label{Eq117}
m_K \left(|x| \right) \leqslant \left|\frac{f(x)-f(0)}{f(1)-f(0)} \right| \leqslant M_K \left( |x| \right)
\end{equation}
Hence
\begin{equation}\label{Eq118}
m_K \left(|x| \right) \leqslant \left| f(x) \right| \leqslant d\,  M_K \left( |x| \right)
\end{equation}
This combined with (\ref{Eq116}) gives a lower bound  of $\delta_f (x,0)$,
\begin{equation}
\delta_f (x,0) = \frac{\Delta_f(x,0)}{|f(x)|} \geqslant \frac{1}{d} \frac{m_K (|x|)}{M_K (|x|)}
\end{equation}
\section{Estimates along integral curves}
Let $\Gamma \subset \mathbb C_0$  be any integral curve of $f\in \mathcal F_K(d)$   parametriced    by  a solution of the differential equation $\dot{x}(t)= f \big( x(t) \big)$, $\alpha \leqslant t \leqslant \beta$. It follows from the previous computation that
\begin{equation}\label{Eq40}
 \frac{\dtext |x|}{\dtext t} = |x| \frac{\dtext \ln |x|}{\dtext t} = |x| \re \frac{\dot{x}}{x} = \Delta_f(x,0) \ge m_K(|x|)
 \end{equation}
This shows that the function $t \to |x(t)|$ is strictly increasing in time, whenever $x(t)$ stays away from the  critical point of $f$. Moreover, we have
\[\frac{\dtext |x|}{\dtext t} = \delta_f(x,0) \left|  \frac{\dtext x}{\dtext t} \right|\]
 Let us now  assume that $\Gamma \subset \mathbb C_r^R$. This means that $r \leqslant |x(t)| \leqslant R$ for all $\alpha \leqslant t \leqslant \beta$, so $\delta_f(x,0) \ge \frac{m_K(r)}{d\, M_K(R)}$ for $x\in \Gamma$. Hence
\begin{eqnarray}
|x(\beta)-y(\alpha)| &\leqslant & \int_\alpha^\beta \left| \dot{x}(t) \right|\, \dtext t  \leqslant    \int_{\alpha}^{\beta} \frac{1}{\delta_f(x,0)} \frac{\dtext |x|}{\dtext t} \dtext t   \\
&=&  d \,  \frac{M_K (R)}{m_K(r)} \big(|x(\beta)| - |x(\alpha)|  \big)  \nonumber
\end{eqnarray}
We just proved  a reverse type triangle inequality along $\Gamma$.
\begin{lemma}
Let $x_1$ and $x_2$ be points in an integral curve of $f$ such that $r\le |x_1|, |x_2| \le R$. Then
\begin{equation}\label{Eq122}
|x_1-x_2| = C_r^R  \cdot \big||x_1| -|x_2| \big|   \, , \quad \quad x_1,x_2 \in \Gamma  \; \footnote{Hereafter, $C_r^R$ stands for a constant depending on the annulus $\mathbb C_r^R$. It  also depends  on the family $\mathcal F_K(d)$, though  we shall not indicate the dependence on  the parameters $K$ and $d$, as the need will not arise.  However, for  the sake of notational simplicity we shall allow $C_r^R$ to vary from line to line.}
\end{equation}
\end{lemma}
Another point of significance is that the time difference between points in $\Gamma \subset \mathbb C_r^R$ is finite.  Indeed, the time between $x(\beta)$ and $x(\alpha)$ can be estimated as follows.
\begin{eqnarray*}
\left| x(\beta)  \right| - \left| x(\alpha)  \right| = \int_\alpha^\beta \frac{\dtext |x|}{\dtext t} \dtext t \ge \int_\alpha^\beta m_K(|x|)\, \dtext t \ge m_K(r) (\beta-\alpha)
\end{eqnarray*}
Hence, whenever $ r \leqslant |x(\alpha)| \leqslant |x(\beta)| \leqslant R$ we have
\begin{equation}\label{Eq124}
  \beta-\alpha \leqslant \frac{\left| x(\beta)  \right| - \left| x(\alpha)  \right|   }{m_K(r)} \le\frac{R-r}{m_k(r)}
\end{equation}
On the basis of these inequalities we  may now prove that the endpoints of   the maximal extension of the local  integral curves are the critical points of $f$.
\begin{corollary}\label{Cor42}
Let   $x=x(t)$ be  a solution to the system $\dot{x}=f(x)$, $f \in \mathcal F_K(d)$  for $t\in (\alpha, \beta)$. Here  $-\infty \leqslant \alpha < \beta \leqslant + \infty $ are the  endpoints of the maximal interval of existence. Then
\begin{equation}\label{Eq125}
\lim\limits_{t \searrow \alpha} x(t)=0 \quad \textnormal{ and } \quad \lim\limits_{t \nearrow \beta} x(t)=\infty
\end{equation}
\end{corollary}
\begin{proof}
As the function $t \to |x(t)|$ is increasing, it follows that
\[0 \leqslant r=  \lim\limits_{t \searrow \alpha} |x(t)| < \lim\limits_{t \nearrow \beta} |x(t)| =R \leqslant \infty \]
Suppose to the contrary  that $r>0$. Then the reverse triangle inequality (\ref{Eq122}) shows that we also have a limit $\lim\limits_{t \searrow \alpha} x(t)=a \ne 0$.   By (\ref{Eq124}) this limit  is attained in finite time.  But then, by virtue of Peano's Existence Theorem, the solution $x=x(t)$ admits an extension beyond $\alpha$ (for  some $t<\alpha$) contradicting maximality of the interval  $(\alpha, \beta)$. The same contradiction follows if one assumes that $R<\infty$
\end{proof}
\begin{lemma}\label{Pr121}
Let  $f \colon \mathbb C \to \mathbb C$  be a (nondegenerate) reduced $K$-quasiconformal field. Consider the integral arcs of the same time-length
\begin{eqnarray*}
\mathfrak{X}&= & \big\{x(t) \colon \dot{x}=f(x) \, , \quad \alpha \leqslant t \leqslant \beta  \big\} \subset \mathbb C_r^R \\
\Upsilon &= & \big\{y(t) \colon \dot{y}=f(y) \, , \quad \alpha \leqslant t \leqslant \beta  \big\} \subset \mathbb C_r^R
\end{eqnarray*}
We assume that time-length   equals the  distance between these arcs,
\[\dist (\mathfrak{X}, \Upsilon) = \beta - \alpha\]
Denote by $x_\alpha = x(\alpha)$,  $x_\beta = x(\beta)$ the endpoints of $\mathfrak X$ and $y_\alpha =y(\alpha)$, $y_\beta =y(\beta)$ the endpoints of $\Upsilon$. Then, for all $x\in \mathfrak X$ and $y\in \Upsilon$ we have
\begin{equation}\label{1210}
\Delta_f (x,y) \leqslant C_r^R   \Delta_f (x_\beta, x_\alpha)
\end{equation}
Moreover,
\begin{equation}\label{1211}
\Delta_f(x_\beta, x_\alpha) \leqslant \left| f(x_\beta)  \right| - \left| f(x_\alpha)  \right| + \frac{\diam{\, ^2} f(\mathfrak X)}{2 \, m_K(r)}
\end{equation}
\begin{equation}\label{Kav49}
\log \frac{|x_\beta - y_\beta |}{|x_\alpha - y_\alpha|} \le C_r^R  \left[  \left| f(x_\beta)  \right| - \left| f(x_\alpha)  \right| + \diam{\, ^2} f(\mathfrak X)  \right]
\end{equation}
\end{lemma}
\begin{proof}
By the three point condition  in (\ref{Eq114}), we have
\begin{eqnarray}\label{1212}
\Delta_f(x,y)&=& M_K \left( \left| \frac{x-y}{x_\beta -y}  \right|\right) \Delta_f (x_\beta, y)\nonumber \\ &\leqslant&  M_K \left( \left| \frac{x-y}{x_\beta -y}  \right|\right)  M_K \left( \left| \frac{x_\beta-y}{x_\beta -x_\alpha}  \right|\right) \Delta_f (x_\beta, x_\alpha)
\end{eqnarray}
We need to estimate the fractions under the function $M_K$; the numerators from above and the denominators from below. For this, choose and fix the time parameters $t,s \in [\alpha, \beta]$  such that
\[\left|x(t) -x(s)\right| = \dist (\mathfrak X, \Upsilon)= \beta -\alpha\]
By  Mean Value Theorem,
\begin{eqnarray}
|x-y|&\leqslant& \left| x-x(t)\right| + \left| x(t)-y(s)\right| + \left| y(s)-y \right| \nonumber \\ &\leqslant &\left| \dot{x} (\xi) \right| (\beta - \alpha) + (\beta -\alpha) + \left| \dot{y} (\zeta) \right| (\beta - \alpha) \nonumber
\end{eqnarray}
for some $\alpha \leqslant \xi , \zeta \leqslant \beta$. On the other hand $|\dot{x}|= |f(x)|\leqslant d \,  M_K(|x|) \leqslant d\,  M_K(R)$. Similarly, $|\dot{y}| \leqslant d\,  M_K(R)$. Thus, we have
\[|x-y|  \leqslant \left[ 1+2 dM_K(R) \right] (\beta -\alpha)  \]
and, in particular,
\[\left|x_\beta-y \right|  \leqslant \left[ 1+2d M_K(R) \right] (\beta -\alpha)  \]
As regards the denominators, we have
\[\left|x_\beta-y \right| \geqslant \dist (\mathfrak X , \Upsilon) = \beta - \alpha \]
and, again by the Mean Value Theorem,
\begin{eqnarray*}  \left|x_\beta-x_\alpha \right| = \left| \dot{x}(\zeta)  \right| (\beta -\alpha)  = \left| f(x_\zeta) \right| (\beta -\alpha)    \geqslant m_K(r)(\beta -\alpha)
\end{eqnarray*}
The inequality (\ref{1210}) is now immediate from (\ref{1212}), simply set
\[C_r^R   =M_K \big(1+2dM_K(R)\big)\cdot M_K \left( \frac{1+2dM_K(R)}{m_K(r)}  \right)  \]

To prove (\ref{1211}) we appeal to the following
\begin{lemma}\label{ApLe33}
Let $A,B,Z$ be vectors of an inner product space, $|Z|=1$. Then
\[\langle A-B,Z\rangle \leqslant |A|-|B| + \frac{|B-\lambda Z|^2}{2\lambda}\]
for all $\lambda >0$.
\end{lemma}
\begin{proof}
\begin{eqnarray*}
\langle A-B,Z\rangle \leqslant |A| - \langle B,Z\rangle \leqslant |A|-|B| + \frac{|B-\lambda Z|^2 - (|B|-\lambda)^2}{2 \lambda}
\end{eqnarray*}
\end{proof}
We take   $X=f(x_\beta)$, $Y=f(x_\alpha)$, $Z= \frac{x_\beta - x_\alpha}{|x_\beta -x_\alpha |}$ and $\lambda = \frac{|x_\beta -x_\alpha|}{\beta-\alpha}$. This gives us the estimate
\begin{equation}\label{Eq1210}
\Delta_f (x_\beta, x_\alpha) \leqslant \left| f(x_\beta) \right|  - \left| f(x_\alpha) \right| + \frac{\left| f(x_\beta)- \frac{x_\beta-x_\alpha}{\beta -\alpha} \right|^2}{2 \left|\frac{x_\beta-x_\alpha}{\beta-\alpha}  \right|}
\end{equation}
The letter term   is handled with the aid of   the  Mean Value Theorem. Precisely, there is $  \xi \in [\alpha, \beta]$ such that
\begin{eqnarray*}
 \frac{\left| f(x_\beta)- \frac{x_\beta-x_\alpha}{\beta -\alpha} \right|}{2 \left|\frac{x_\beta-x_\alpha}{\beta-\alpha}  \right|}  = \frac{\left| f(x_\beta)- \dot{x}(\xi) \right|^2}{2 \left| \dot{x}(\xi) \right|}= \frac{\left| f(x_\beta) -f(x_\xi)\right|^2}{2\left| f(x_\xi) \right|}   \leqslant  \frac{\diam^{\, 2} f(\mathfrak X)}{2\,  m_K \left( r\right)}
\end{eqnarray*}
as desired.  The proof of (\ref{Kav49}) proceeds as follows
\begin{eqnarray*}
\log \frac{|x_\beta-y_\beta|}{|x_{\alpha } - y_{\alpha}|} &=& \int_{\alpha}^{\beta} \frac{\dtext }{\dtext t} \log |x(t)-y(t)| \, \dtext t \nonumber \\ &=& \int_{\alpha}^{\beta} \left\langle  \dot{x}(t)- \dot{y}(t) , \, \frac{x(t)-y(t)}{|x(t)-y(t)|^2}  \right \rangle \, \dtext t \\
&=& \int_{\alpha}^{\beta} \frac{\Delta_f \big(x(t), y(t)  \big)}{|x(t)-y(t)|}\, \dtext t  \le C_r^R \Delta_f (x_\beta, x_\alpha)\\
&\leqslant & C_r^R   \left[ \left| f(x_\beta)  \right|  - \left| f(x_{\alpha})  \right|   + \frac{\diam^{\, 2} f( \mathfrak X)}{2m_K(r)} \right]
\end{eqnarray*}
Here, for the inequality before the last, we estimated the numerator in the integrand by   (\ref{1210}) while for the denominator we observed
\[|x(t)-y(t)|  \ge \dist (\mathfrak X , \Upsilon) = \beta -\alpha\]
Then (\ref{Kav49}) follows from (\ref{1211}). The proof of Lemma \ref{Pr121} is complete.
\end{proof}

\section{Quadratic variation along $C^1$-arcs}\label{quadraticsection}

A parametric curve in $\R^n$ is a continuous function $x=x(t)$ defined in an interval $I$  (bounded or unbounded) with values in $\R^n$.   The orientation of a parametric curve is given in the direction of increasing parameter. If $x: I \to \R^n$ is one-to-one, then $x=x(t)$ is called a simple parametric curve; it is called an arc if $I=[\alpha ,\beta]$ is closed and bounded, in which case $x(\alpha)$ and $x(\beta)$ are called the left and the right endpoints. Let $\Gamma= \{x(t) \colon \alpha \le t \le \beta \}$.  A partition of parameters    $\alpha=t_0 < t_1 < \cdots < t_N =\beta$ gives rise to a partition of the curve $\Gamma$, with    $x_j=x(t_j)$, $j=0,1, ..., N$,  called partition points of  $\Gamma$. Furthermore, to every interval $[t_{j-1}, t_j]$ there corresponds a subarc $\gamma_j = x[t_{j-1}, t_j]$ in $\Gamma$.   The arc length of   $\Gamma$ is   denoted by
$|\Gamma|$.

Recall that $p$-variation, $p \ge 1$, of a continuous map $f \colon \Omega \to \R^n$ along a compact $C^1$-arc $\Gamma \subset \Omega$ is defined by
\begin{equation}
\left| f(\Gamma) \right|_p = \sup \left( \sum_{\nu=1}^N \left| \diam f(\gamma_\nu) \right|^p  \right)^\frac{1}{p} < \infty
\end{equation}
where the supremum runs over all finite partitions of $\Gamma$ into subarcs $\gamma_1, \gamma_2 , \cdots , \gamma_N$. Note that
\[\left| f(\Gamma) \right|_p \le \left| f(\Gamma) \right|_q\, , \quad \quad \textnormal{ when } 1 \le q \le p\]
When $p=1$ we recover the classical concept of bounded variation, and denote $|f(\Gamma)|_1 = |f(\Gamma)|$ for simplicity.

The quadratic variation along $C^1$-arcs of any homeomorphism $f \colon \mathbb C \to \mathbb C$ in $W^{1,2}_{\loc}(\mathbb C)$ is finite, see \cite[Theorem 4.3]{Ma}.  We shall demonstrate this property, together with specific bounds, for planar $K$-quasiconformal mappings.
\begin{theorem}\label{Th81}
Let $f \colon \mathbb C \to \mathbb C$ be a $K$-quasiconformal mapping and $\Gamma$ a $C^1$-arc in $\mathbb C$. Then $\left|f(\Gamma)\right|_2 < \infty$. If, moreover, $f \in \mathcal F_K(d)$ and $\Gamma$ lies in the annulus $\mathbb C_r^R$, then
\begin{equation}\label{Eq815}
\left| f(\Gamma) \right|_2 \leqslant C^R_r   \diam f(\Gamma)
\end{equation}
\end{theorem}
\begin{proof}
 Let $z=z(\tau)$, $\alpha \leqslant \tau \leqslant \beta$, be the arc-length parametrization of $\Gamma$; that is, $|\dot{z} (\tau)| \equiv 1$ and $|\Gamma| = \beta-\alpha$. The $C^1$-modulus of regularity of $\Gamma$ is defined by
\begin{equation}\label{Eq82}
\Lambda (\tau) = \sup \left\{ |\dot{z}(t) - \dot{z}(s)| \colon  \alpha\leqslant t , \, s \leqslant \beta , \, |t-s| \leqslant \tau  \right\}
\end{equation}
Clearly, the function $\Lambda \colon [0, |\Gamma|] \to [0,2]$   is continuously nondecreasing and $\Lambda (0)=0$. By the definition, we have
\begin{equation}
\left| \dot{z}(t)- \dot{z}(s)  \right| \leqslant \Lambda (t-s) \, , \quad  \textnormal{ for } \alpha \leqslant s \leqslant t \leqslant \beta
\end{equation}
We first consider short arcs, assuming that $\Lambda (|\Gamma|)  \leqslant \frac{\sqrt{2}}{2}$, or equivalently,
\begin{equation}\label{Eq84}
\left| \dot{z}(t)- \dot{z}(s)  \right| \leqslant \frac{\sqrt{2}}{2} \quad \textnormal{ for all } \alpha \leqslant s \leqslant t \leqslant \beta
\end{equation}
{\bf Claim.}
{\it Under the assumption (\ref{Eq84}) we have
\begin{equation}\label{Eq85}
\left| f(\Gamma) \right|_2 \leqslant C_K \diam f(\Gamma)
\end{equation}
where $C_K$ depends only on the distortion $K$ of the mapping $f$.}\\ \\
{\it Proof of Claim.}
With the aid of a rigid motion   we place $\Gamma$ into a position in which its endpoints  are real numbers, say the left endpoint is   the origin and the right endpoint is a positive number $L$.
By  the Mean Value Theorem, there exists a middle point $\zeta \in [\alpha,\beta]$ such that
\[1 = \frac{z(\beta)-z(\alpha)}{|z(\beta)-z(\alpha)|}= \dot{z}(\zeta )\]
Then, in view of Condition (\ref{Eq84}),
\begin{equation}\label{Eq86}
|1- \dot{z}(\xi)| = \left| \dot{z} (\zeta) - \dot{z} (\xi)   \right| \leqslant \frac{\sqrt{2}}{2} \quad \textnormal{ for every } \xi \in [\alpha,\beta]
\end{equation}
This,  by the Mean Value Theorem again, yields
\begin{equation}\label{Eq87}
\left|1- \frac{z(t)-z(s)}{|z(s)-z(t)|}  \right| \leqslant \frac{\sqrt{2}}{2} \quad \textnormal{ for  }  \alpha \leqslant s < t \leqslant \beta
\end{equation}
Then  (\ref{Eq86})  combined with the identity $\left| \dot{z}(\xi) \right| =1$ gives
\begin{equation}\label{Eq88}
\left| \im \dot{z}(\xi)  \right| \leqslant \frac{\sqrt{7}}{3} \re \dot{z}(\xi) \leqslant \re \dot{z}(\xi)
\end{equation}
for every $\alpha \leqslant \xi \leqslant \beta$. In other words, the function $t \to \re z(t)$ is strictly increasing from $0$ to $L$. In particular, $\Gamma$ becomes a graph of a function over the interval $[0,L]$.
Given any parameters $\alpha \leqslant t, s \leqslant \beta$, by   Cauchy's Mean-Value Theorem, we have
\[ \frac{\im z(t)  -\im z(s) }{\re z(t) -\re z(s)}= \frac{\im \dot{z} (\xi)}{\re \dot{z} (\xi)} \]
for some $\xi \in [s,t]$. If we combine this with (\ref{Eq88}), we obtain
\begin{equation}\label{Eq89}
\left| \im z(t)- \im z(s)  \right| \leqslant \re z(t) - \re z(s)
\end{equation}
In particular, letting $s=\alpha$, we see that
\begin{equation}\label{Eq810}
\left| \im z(t) \right| \leqslant \re z(t) \leqslant L
\end{equation}
Next, we  choose and fix an   arbitrary partition points of $\Gamma$. Denote them by  $0=z_0, z_1 , \cdots , z_{N-1}, z_N=L$.  We consider the rectangles
\[\mathcal R_j= \big\{ z \colon \re z_{j-1} \le \re z \le \re z_j \, , \;  -L \le \im z \le L  \big\}\]
and the subarcs of $\Gamma$,
\[\gamma_j = \mathcal R_j \cap \Gamma \, , \quad j=1,2, ..., N\]
Inequality (\ref{Eq89}) shows that
\[ \underset{z\in \gamma_j}{\max} \{\im z\}\,  -\,  \underset{z\in \gamma_j}{\min} \{ \im z \}  \le \re z_j - \re z_{j-1} \]
This  in turn allows us to confine each arc $\gamma_j$ in a square $Q_j \subset \mathcal R_j$. Such squares are mutually disjoint and   lie in a square $Q$  centered at $0$ and of side length $2L$
\[Q= \big\{ z \colon -L \le \re z \le L \, , \quad -L \le \im z \le L   \big\}\]
 We note that $\Gamma$ joins $\partial Q$ with the center of $Q$. Therefore
\[   \underset{z\in \partial Q}{\min} \left|  f(z)-f(0) \right| \le \diam f(\Gamma)  \]
On the other hand, since $f$ is $K$-quasiconformal, we have
\begin{eqnarray*}
\left|f(Q)\right| \le \pi \, \underset{z\in \partial Q}{\max} \left| f(z)-f(0)  \right|^2 \le C_K \underset{z\in \partial Q}{\min} \left| f(z)-f(0)  \right|^2 \le C_K \diam^{\, 2} f(\Gamma)
\end{eqnarray*}
We also have the reverse type estimates for the squares $Q_j$, namely
\[\diam^{\, 2} f(Q_j) \le C_K \left| f(Q_j) \right|\]
This is a consequence of $K$-quasiconformality of $f$, as well.

Now we are ready to  estimate the quadratic variation of $f$ along $\Gamma$. As cubes $Q_j$ are mutually disjoint, we have
\begin{eqnarray}
\sum_{j=1}^N  \diam^{ \, 2} f(\gamma_j) &\leqslant & \sum_{j=1}^N \diam^{\, 2} f(Q_j)  \leqslant C_K \sum_{j=1}^N \left|  f(Q_j)\right| \nonumber \\
&=& C_K \left| f \left( \cup Q_j \right)  \right| \leqslant C_K \left|f(Q ) \right| \leqslant C_K  \diam^{\, 2} f(\Gamma)
\end{eqnarray}
This completes the proof of our Claim. \vskip0.3cm

To estimate $|f(\Gamma)|_2$ for long $C^1$-arcs, we partition  $\Gamma$ into $\ell $ disjoint subarcs, say $\Gamma = \Gamma_1 \cup \cdots \cup \Gamma_\ell $, where $\ell$ is large enough to ensure condition (\ref{Eq84}) on each subarc. We fix this partition.  Now, let $\gamma_1, \gamma_2, ..., \gamma_N$ be any partition of $\Gamma$ into subarcs $\gamma_j$, $1 \leqslant j \leqslant N$, to be used for computing the quadratic variation of $f$ along $\Gamma$. There are two kinds of subarcs in this partition. The first kind of the subarcs, denoted by $\gamma_j^\prime$, are those which  lay entirely in one of $\Gamma_1, ..., \Gamma_\ell$. Certainly, using (\ref{Eq85}),  we have
\[\sum \diam^{\, 2} f(\gamma_j^\prime) \leqslant \sum_{\nu =1}^\ell \left| f(\Gamma_\nu )  \right|_2^2 \leqslant \ell \,  C_K \diam^{\, 2} f(\Gamma)   \]
Then, there are  at most $\ell$ remaining subarcs, denoted by $\gamma_j^{\prime \prime}$. Each of them contains at least one endpoint of the partition $\Gamma=\Gamma_1 \cup \cdots \cup \Gamma_\ell$.  For these subarcs we have trivial estimate
\[\sum \diam^{\, 2} f (\gamma_j^{\prime \prime}) \leqslant \sum  \diam^{\, 2} f(\Gamma) \le \ell \,  \diam^{\, 2} f(\Gamma)  \]
In summary
\[\sum_{j=1}^N \diam^{\, 2} f (\gamma_j) \leqslant \ell (1+C_K) \diam^{\, 2} f(\Gamma)< \infty\]
Since the partition $\gamma_1, ..., \gamma_N$ of $\Gamma$ was chosen arbitrarily, it follows that
\begin{equation}\label{Eq814}
\left| f(\Gamma) \right|_2^2 \leqslant \ell (1+C_K) \diam^{\, 2} f(\Gamma) < \infty
\end{equation}
as desired.

When $\Gamma \subset \mathbb C_r^R$ is an integral curve of a reduced $K$-quasiconformal field this   estimate   lets us deduce specific bound of the quadratic variation. \\
{\it Proof of (\ref{Eq815}).} Let $z=z(\tau)$ be the arc-length parametrization of $\Gamma$; that is, $\dot{z}(\tau) = \frac{f(z(\tau))}{|f(z(\tau))|}$. We aim to partition $\Gamma$ into subarcs $\Gamma_1, ..., \Gamma_\ell$ so that
\begin{equation}\label{Eq816}
\left|    \frac{f(a)}{|f(a)|}- \frac{f(b)}{|f(b)|} \right|  \leqslant \frac{\sqrt{2}}{2} \, , \quad \textnormal{whenever }a,b \in \Gamma_\nu
\end{equation}
For this, we observe that
\begin{eqnarray*}
\left|    \frac{f(a)}{|f(a)|}- \frac{f(b)}{|f(b)|} \right|  & = &  \left|    \frac{f(a)- f(b)}{|f(a)|}- \frac{f(b)}{|f(b)|} \frac{|f(a)|-|f(b)|}{|f(a)|} \right| \\
&\leqslant &  \frac{2\, |f(a)-f(b)|}{|f(a))| }\leqslant 2 M_K \left( \left|\frac{a-b}{a} \right| \right) \\ &\leqslant &2 M_K \left( \frac{a-b}{r}\right) \leqslant \frac{\sqrt 2}{2}
\end{eqnarray*}
provided $|a-b| \leqslant \epsilon_Kr$, where $\epsilon_K$ is determined from the equation  $M_K(\epsilon _K) = \frac{\sqrt{2}}{4}$. In view of   the reverse triangle inequality (\ref{Eq122}) it suffices to  make the partition of $\Gamma$ fine enough to satisfy
\begin{equation}\label{Kav514}
  \big||b|-|a|\big| \leqslant \frac{ \epsilon_K \cdot r}{C_r^R} \quad \textnormal{for } a,b \in \Gamma_\nu, \;  \nu =1,2, ..., \ell
\end{equation}
To this effect we divide the annulus $\mathbb C^R_r$ into $\ell$ annuli  $\mathbb C_{r_1}^{R_1}$, ..., $\mathbb C_{r_\ell}^{R_\ell}$, each of   width $R_\nu -r_\nu = \frac{1}{\ell} (R-r)$. Inequality (\ref{Kav514}) yields a sufficient lower bound for $\ell$.
\[\ell \ge \frac{(R-r) C_r^R }{r   \epsilon_K}\]
Finally, we notice that  the intersections  $\Gamma_\nu=\Gamma \cap \mathbb C^{R_\nu}_{r_\nu}$ are subarcs of $\Gamma$, because the function $t \to |x(t)|$ is strictly increasing along $\Gamma$. Inequality (\ref{Eq816}) now holds for all $a,b \in \Gamma_\nu$, $\nu=1,2, ..., \ell$. Substitute this integer value $\ell=\ell (r,R,K, d)$ into (\ref{Eq814}) to conclude with (\ref{Eq815}).
\end{proof}

\section{A partition of two curves}
\begin{lemma}\label{ApLe1}$\textnormal{\textsc{(Partition Lemma)}}$ Suppose we are given two continuous functions $x,y \colon [-\infty, t_0] \to \R^n$ such that $x(-\infty) = y(-\infty)$, whereas $x(t_0) \ne y(t_0)$. Then there exists (unique) sequence $t_0>t_1> \cdots > t_k >t_{k+1} \cdots \to t_\infty \geqslant - \infty$, such that
\begin{equation}\label{ApEq1}
\underset{t_{k+1} \leqslant t,s \leqslant t_k }{\inf} |x(t)-x(s)| = t_k - t_{k+1}\, , \quad x(t_\infty)=y(t_\infty)
\end{equation}
If, in addition,   $x$ and $y$ are continuously differentiable, then for every $\tau \in [t_{k+1}, t_k]$ we have
\begin{equation}\label{ApEq2}
\left| x(t_k) -y(t_k) \right| \leqslant (1+C) \, |x(\tau) - y(\tau)|
\end{equation}
where
\begin{equation}\label{ApEq3}
C= \underset{\tau \leqslant t \leqslant t_k}{\sup} \left( \left| \dot{x}(t) \right| + \left| \dot{y}(t) \right|  \right) \; \footnote{It will be important for the use of (\ref{ApEq2}) that the supremum in (\ref{ApEq3}) runs over the interval $[\tau, t_k]$, not over the entire interval $[t_{k+1},t_k]$.}
\end{equation}
\end{lemma}
\begin{proof}
We construct such  sequence $\{t_k\}$ by induction. Suppose we are given $t_0 >t_1 \cdots >t_k > - \infty$, $x(t_k) \ne y(t_k)$.  Consider the following function
\[\varphi_k (\tau) = \tau + \underset{\tau \leqslant t,s \leqslant t_k }{\inf} |x(t)-y(s)|\, , \quad \tau \in [-\infty, t_k] \]
Clearly, $\varphi_k$ is continuous and strictly increasing. Before we make the induction step, let us think of $k$ to be equal zero. Since $\varphi_k (t_k) >t_k$ and $\varphi_k (-\infty)=-\infty$, we find (unique) parameter $t_{k+1}<t_k$ such that $\varphi_k (t_{k+1})=t_k$. This means that
\[\underset{t_{k+1} \leqslant t,s \leqslant t_k }{\inf} |x(t)-y(s)| = t_k - t_{k+1}\]
In particular, $x(t_{k+1}) \ne y(t_{k+1})$. Now, the same reasoning provides  for  the induction step. We then obtain the desired decreasing sequence
\[ t_0 > t_1 > \cdots  > t_k >t_{k+1} \cdots \to t_\infty \geqslant -\infty  \]
Finally, if $t_\infty = - \infty$, then  $x(t_\infty) = x(-\infty) = y(-\infty) = y(t_\infty)$
If, however,  the sequence $\{t_k\}$ is converging to some finite number $t_\infty$, then by (\ref{ApEq1}) we conclude that $x(t_\infty)=y(t_\infty)$.

The proof of (\ref{ApEq2}) is a matter of triangle inequality combined with (\ref{ApEq1}).
\begin{eqnarray*}
\left|  x(t_k) -y(t_k)  \right| &\leqslant & \left|  x(\tau) -y(\tau)  \right| + \left|  x(t_k) -x(\tau) +y(\tau)-y(t_k)  \right| \\
&=& \left|  x(\tau) -y(\tau)  \right|  + \left| \int_\tau^{t_k}  \left[ \dot{x}(t) - \dot{y}(t) \right] \, \dtext t  \right| \\
&\leqslant & \left|  x(\tau) -y(\tau)  \right| + C \, \left|  t_k - \tau \right|    \\
&\le& |x(\tau)-y(\tau)|+ C |t_k -t_{k+1}| \\
&=& \left|  x(\tau) -y(\tau)  \right| + C  \underset{t_{k+1} \leqslant t,s \leqslant t_k }{\inf} |x(t)-y(s)| \\
&\leqslant & (1+C) \,  \left|  x(\tau) -y(\tau)  \right|
\end{eqnarray*}
as claimed.
\end{proof}

\section{Uniqueness, proof of Theorem \ref{ThMain}}
We have already established (\ref{EqI111}) and (\ref{EqI112}) by Corollary \ref{Cor42} and Inequality \ref{Eq40}. To complete the proof of Theorem \ref{ThMain} it remains to establish uniqueness of the local solutions. Let us state this task  explicitly:
\begin{proposition}\label{Th131}
Suppose we are given two local solutions of the differential system
\[\dot{x}(t)=f \big(x(t)\big)  \quad \textnormal{ and } \quad \dot{y}(t)=f \big(y(t)\big) \quad \mbox{ for } t \in (\alpha, \beta) \]
where $f\in \mathcal F_K(d)$ and $x(t^\prime) = y(t^\prime) \ne 0$ for some $t^\prime \in (\alpha, \beta)$. Then $x(t)=y(t)$ for all $t\in (\alpha, \beta)$.
\end{proposition}
\begin{proof}
The equality $x(t)=y(t)$ for $x \in (\alpha, t^\prime ]$ is immediate since the function $t \to |x(t)-y(t)|$ is nondecreasing. Suppose, to the contrary, that $|x(t)-y(t)| \not\equiv 0$. Thus, there exists $s\in [t^\prime , \beta)$ such that $x(s)=y(s)$ and $|x(t)-y(t)|>0$ for all $t\in (s, \beta)$. For notational convenience we can certainly assume that $s =0$. Therefore, $x(t) \ne y(t)$ for all $0<t< \beta$ and the common value $x(0)=y(0)$ is not the critical point of $f$. Choose and fix $t_0 \in (0, \beta)$. Thus we have
\[ x(t), y(t) \in \mathbb C^R_r \quad \textnormal{ for } 0 \leqslant t \leqslant t_0  \]
where we define \[r=|x(0)|=|y(0)|  \quad \mbox{ and } \quad R= \max \big\{|x(t_0)|, |y(t_0)|  \big\}\]

We shall make use of the partition
\[ t_0 >t_1 > \cdots t_k >t_{k+1} \to 0\]
as in Lemma \ref{ApLe1}. Accordingly,
\[\dist \{ x[t_{k+1}, t_k] , \, y[t_{k+1}, t_k]  \}= t_k-t_{k+1}\, , \quad k=0, 1, 2, ...   \]
Denote by $x_k$ and $y_k$ the values of $x$ and $y$ at time $t_k$, respectively. We also denote by $\gamma_k$ the arc $\gamma_k =\{x(t) \colon t_{k+1} \leqslant t < t_k  \}$. Then, in view of  Inequality (\ref{Kav49}) in Lemma \ref{Pr121}, we obtain
\begin{eqnarray*}
\log \frac{|x_k-y_k|}{|x_{k+1} - y_{k+1}|}  \leqslant  C_r^R   \left[ \left| f(x_k)  \right|  - \left| f(x_{k+1})  \right|   + \diam^{\, 2} f(\gamma_k)  \right]
\end{eqnarray*}
The telescoping structure of   the terms in this inequality  helps  us to sum them up, with substantial  cancellations. Summing with respect to  $k=0,1, 2, ... , m-1$,  the surviving terms are:
\begin{equation}\label{Eq131}
\log \left| \frac{x_0-y_0}{x_m -y_m} \right| \leqslant C_r^R  \left[ \left| f(x_0)  \right|  - \left| f(x_{m})  \right|   +  |f(\Gamma)|_2^2  \right] \le C_r^R
\end{equation}
where   $|f(\Gamma)|_2$ stands for the quadratic variation of $f$ along $\Gamma=\{x(t) \colon 0 \leqslant t \leqslant t_0\}$. By Theorem \ref{Th81},  the right hand side of (\ref{Eq131}) is bounded by a constant $C_r^R$ independent of $m$. However, the left hand side increases to $+\infty$ as $m \to -\infty$, because $x_m -y_m = x(t_m)-y(t_m) \rightarrow x(0)-y(0)=0$. This contradiction proves Theorem \ref{Th131}
\end{proof}
\section{Proof of Theorem \ref{ThMain2}}
First we prove Inequality (\ref{EqI113}). Let $(\alpha_1, \beta_1)$ denote the maximal interval of existence of the solution $x=x(t)$ of the system $\dot{x}=f(x)$ in $\mathbb C_0$, as in Theorem \ref{ThMain}. It will be convenient to view $x$ as a solution in $\mathbb C$ defined in the interval $[-\infty , \beta_1)$, by setting $x(t)=0$ for $-\infty \le t \le \alpha_1$. The extended solution is a continuous function $x \colon [-\infty , \beta_1 ) \to \mathbb C$. Now  consider another extended solution  $y \colon [-\infty , \beta_2) \to \mathbb C$. Suppose that at some time $t_0 < \min \{\beta_1 , \beta_2\}$ we have $x(t_0) \ne y(t_0)$. In particular, $t_0 \ne - \infty$. We make use of  a decreasing sequence
\[t_0 > t_1 > \cdots t_k > t_{k+1} \to t_\infty \]
as in Lemma \ref{ApLe1}. With the same arguments as were used in  (\ref{Eq131}) we obtain $|x(t_0)-y(t_0)| \le C^R_r |x(t_k)-y(t_k)|$. Then the complementary inequality (\ref{ApEq2}) yields
\[\left| x(t_0) - y(t_0)  \right|  \le C_r^R \left| x(t_k) - y(t_k)  \right| \le C_r^R \left| x(\tau) - y(\tau)  \right| \]
provided
\[r \le |x(\tau)| \le |x(t_0)| \le R\]
and
\[r \le |y(\tau)| \le |y(t_0)| \le R\]
We just proved Inequality (\ref{EqI113}).  Now (\ref{EqI1122}) is immediate.
\begin{eqnarray*}
|x(t)-y(s)| &\le& |x(t)-x(s)| + |x(s)-y(s)|\\ &\le& \left| \dot{x}(\xi) \right| |t-s| + C^R_r |x(0)-y(0)|\\
&\le & M_K (R)|t-s| + C_r^R |x(0) -y(0)|
\end{eqnarray*}

\section{Bi-Lipschitz continuity of $\Phi$}\label{BLsection}
The purpose of this section is to prove the inequality (\ref{EqI15}). For this we represent $\Phi$ in quasipolar coordinates. Then (\ref{EqI15}) is the same as
\begin{theorem}\label{Th141}
Given two points $z_1=(\rho_1, e^{i \theta_1})$ and $z_2=(\rho_2, e^{i \theta_2})$ in $\mathbb C^R_r$, we have
\begin{equation}\label{Eq7st}
c_r^R  \left| z_1-z_2\right|  \leqslant   \left| \rho_1 e^{i\theta_1} -  \rho_2 e^{i\theta_2} \right|  \leqslant  C_r^R  \left| z_1-z_2\right|
\end{equation}
\end{theorem}
\begin{proof}
First we prove the following
\begin{lemma}\label{Lem141}
Given two integral arcs $\mathfrak X, \Upsilon \subset \mathbb C_r^R$ and points $a\in \mathfrak X$, $b \in \Upsilon$ such that $|a|=|b|= \rho$. Then
\[|a-b| \leqslant C_r^R \dist (\mathfrak X , \Upsilon)\]
In other words,
\begin{equation}\label{Eq7stst}
|a-b| \leqslant C_r^R\,  |x_0-y_0|\, , \quad \textnormal{ whenever } x_0 \in \mathfrak X \textnormal{ and } y_0 \in \Upsilon
\end{equation}
\end{lemma}
\begin{proof}
{\bf Case 1.} Suppose $x_0$ and $y_0$ lie in the opposite side of the circle $\mathbb S_\rho=\{z \colon |z|=\rho\}$. For example,  $|x_0| \leqslant \rho \leqslant |y_0|$. Then,
\begin{eqnarray*}
|a-b| &\leqslant & |a-x_0|+|x_0-y_0|+|y_0-b| \\
&\leqslant & C_r^R \big( |a|-|x_0|  \big) + |x_0 -y_0| + C_r^R \big( |y_0|-|b| \big) \\
&=& C_r^R  \big( |y_0|-|x_0| \big) +|y_0-x_0| \\
&\leqslant & (1+ C_r^R) |y_0-x_0|
\end{eqnarray*}
Here in the second line  we have used the reverse triangle inequality \ref{Eq122}.\\
{\bf Case 2.} Suppose both $x_0$ and $y_0$ lie inside $\mathbb S_\rho$. We use time parametrization for $\mathfrak X$ and $\Upsilon$, with $t=0$ as starting time for $x_0= x(0)$ and $y_0=y(0)$. Therefore $a=x(t)$ and $b=y(s)$ for some parameters $t$ and $s$.  We have $|x(t)|=|y(s)|$. Since the functions $t \to |x(t)|$ and $s \to |x(s)|$ are increasing,   it follows that $t \geqslant 0$ and $s \geqslant 0$. We may assume without loss of generality that $0 \leqslant s \leqslant t$.  Thus the point $x(s)$ lies in $\mathfrak X$.  Clearly,
\[ r \leqslant |x_0| =|x(0)| \leqslant |x(s)| \leqslant |x(t)| =|a| \leqslant \rho \]
Hence,  $x(s)\in \mathfrak X \cap \mathbb C^R_r$. Now  using the reverse triangle inequality (\ref{Eq122})  we obtain
\begin{eqnarray*}
|a-b| &=& |x(t)-y(s)| \leqslant |x(t)-x(s)|+|x(s)-y(s)| \\
&\leqslant & C_r^R \big( |x(t)|-|x(s)| \big) + |x(s)-y(s)|   \\
&=& C_r^R \left(|y(s)|-|x(s)|\right) +|x(s)-y(s)|\\
&\leqslant & (1+C_r^R)\,   |x(s)-y(s)|\\
 &\le  & (1+C_r^R)\, C_r^R \, |x(0)-y(0)|
\end{eqnarray*}
where in the last step we have appealed to (\ref{EqI113}).

In much the same way we prove (\ref{Eq7stst}) when both $x_0=x(0)$ and $y_0=y(0)$ lie outside $\mathbb S_\rho$. The only difference is that the parameters $t$ and $s$ will be negative.
\end{proof}
{\it Proof of Theorem \ref{Th141}.} Obviously, we have
\[| \rho_1 - \rho_2 | = \big|  |z_1|-|z_2| \big| \leqslant |z_1-z_2|  \]
Denote by $\mathfrak X , \Upsilon \subset \mathbb C_r^R$ the integral arcs which intersect the unit circle at the points $a=e^{i \theta_1}$ and $b=e^{i \theta_2}$, respectively. Thus $z_1 \in \mathfrak X$ and $z_2 \in \Upsilon$. By Lemma \ref{Lem141}, we have
\[ \left| e^{i \theta_1}- e^{i \theta_2}  \right| =|a-b| \leqslant C_r^R \dist ( \mathfrak X , \Upsilon)  \leqslant C_r^R |z_1-z_2| \]
These two inequalities prove the estimate in the right hand side of (\ref{Eq7st}).  For the opposite estimate we choose two points $z_1= (\rho_1 , e^{i \theta_1}) \in \mathfrak X$ and $ z_2= (\rho_2 , e^{i \theta_2}) \in \Upsilon$, where $r \leqslant \rho_1, \rho_2 \leqslant R$. Define $a=z_1 = (\rho_1, e^{i \theta_1}) \in \mathfrak X$ and  $b = (\rho_1, e^{i \theta_2}) \in \Upsilon$. These are points of the same distance from the origin, $|a|=|b| = \rho_1$. By Lemma \ref{Lem141}, we have
\[ |z_1 -b| =|a-b| \leqslant   C_r^R \dist ( \mathfrak X , \Upsilon)  \leqslant C_r^R \left| e^{i\theta_1} - e^{i \theta_2} \right|  \]
On the other hand, $b$ and $z_2$ belong to the same integral arc in $\mathbb C^R_r$, so by the reverse triangle inequality (\ref{Eq122})
\[|b-z_2| \leqslant C_r^R \big| |b|- |z_2|  \big| = C^R_r \left| \rho_1 - \rho_2 \right| \]
Summing up the above inequalities we obtain
\begin{eqnarray*}
|z_1-z_2| &\le & |z_1-b|+|b-{z_2}| \leqslant C_r^R \left(  |\rho_1 - \rho_2| + \left| e^{i \theta_1} - e^{i \theta_2}  \right|  \right) \\ &\le & C_r^R   \left| \rho_1 e^{i\theta_1} -  \rho_2 e^{i\theta_2} \right|
\end{eqnarray*}
as claimed. Theorem \ref{Th141} is proved.
\end{proof}
Inequality \ref{EqI15} tells us that $\Phi$ and its inverse, denoted by $\Psi = \Phi^{-1} \colon \mathbb C \to \mathbb C$, are locally Lipschitz on $\mathbb C_0$.  Therefore, both  $\Phi$ and $\Psi$ are differentiable almost everywhere.
\section{Polar equation for integral curve $\Gamma_\theta$}
The points $z\in \Gamma_\theta$ can be represented by the equation
\begin{equation}
z= \rho e^{i \varphi (\rho)}\, , \quad \quad 0< \rho < \infty
\end{equation}
where $\varphi$ solves the initial value problem
\begin{equation}
\begin{cases} \dot{\varphi} (\rho) = F(\rho , e^{i \varphi}) = \displaystyle \frac{1}{\rho} \frac{\im e^{-i\varphi} f \left( \rho e^{i \varphi}\right)}{ \re e^{-i\varphi} f \left( \rho e^{i \varphi}\right) }  \\
\varphi(1)= \theta
\end{cases}
\end{equation}
The scalar function $F: \mathbb R_+ \times \mathbb S^1 \to \mathbb R$ can be found  as follows
\begin{eqnarray*}
f(z)&=& \frac{\dtext z}{\dtext t} = \frac{\dtext z}{\rho} \cdot \frac{\dtext \rho}{\dtext t} = e^{i \varphi} \left( 1+i \rho \dot{\varphi}  \right) \cdot  \frac{\dtext |z(t)|}{\dtext t} \\
&=& \left( 1+i \rho \dot{\varphi} \right) z \cdot \re \frac{f(z)}{z}
\end{eqnarray*}
Hence,
\[\rho \, \dot{\varphi} (\rho) = \frac{\im \frac{f(z)}{z}}{\re \frac{f(z)}{z}}\]
and
\begin{equation}
F(\rho, e^{i \varphi}) =   \frac{1}{\rho} \frac{\im e^{-i\varphi} f \left( \rho e^{i \varphi}\right)}{ \re e^{-i\varphi} f \left( \rho e^{i \varphi}\right) }
\end{equation}
The single equation just established for $\varphi$ is no longer  autonomous.  But it can be  useful   for a discussion of geometric properties of the integral curves.
\section{Integrating factor, Proof of Theorem \ref{ThMain3}}
Every complex number $z \ne 0$ has infinitely many quasipolar angles which differ from each other by multiple of $2 \pi$. These are real numbers $\theta \in \mathbb R$ such that the integral curve through the point $e^{i \theta} \in \mathbb S^1$ contains $z$. We denote by $\Theta =  \Theta (z)$
the multivalent function that assigns to $z$ its quasipolar angles. It is worth pointing out that $\Theta (z)$ has a continuous branch on every simply connected domain $\Omega \subset \mathbb C_0$. Two such branches differ by a constant. In other words, it makes sense to speak of the gradient of $\Theta$,
\[ \nabla \Theta (z) = \left(  \frac{\partial \Theta}{\partial x}\, ,\,  \frac{ \partial \Theta}{\partial y}  \right) \]
where we have chosen a continuous branch of $\Theta$ near the given point $z\in \mathbb C_0$. By the definition of the map $\Phi \colon \mathbb C_0 \to \mathbb C_0$, we have the identity
\[ e^{i \Theta  } = \frac{\Phi (z)}{|\Phi (z)|}  \]
Inequality (\ref{EqI15})   shows that for   almost every $z\in \mathbb C_0$
\begin{equation}
0 < \mathfrak m \big(|z|\big) \leqslant \left|  \nabla \Theta (z)\right| \leqslant \mathfrak M \big( |z|\big) < \infty
  \end{equation}
Here the lower and upper bounds are given by continuous functions. Of course these functions blow up at the critical point $z=0$ and at $z=\infty$. Any continuous branch of $\Theta$ along an integral curve  $\Gamma_\theta$ is constant, namely $\Theta \equiv \theta + 2 \pi k$. We differentiate $\Theta (z)$ along $\Gamma_\theta$ to obtain
\[  \quad  f(z) \, \Theta_z + \overline{f(z)}\,  \Theta_{\bar z} =0  \quad \mbox{ where } \Theta_{\bar z} = \overline{\Theta_z}  \]
Hence
\begin{equation}\label{Eq172}
if(z)= \pm \frac{|f|}{|\Theta_z|} \Theta_{\bar z}
\end{equation}
We then see that the vector field
\begin{equation}\label{Eq173}
V(x,y)=if(x+iy)
\end{equation}
takes the form
\begin{equation}\label{Eq174}
V(x,y)= \lambda   \nabla \Theta \quad \quad \textnormal{ with } \lambda (z)= \pm \frac{|f(z)|}{2 |\Theta_z (z)|}
\end{equation}
It is not difficult to see that the sign is plus. Indeed, since $\Theta$ is increasing in the direction counterclockwise on every circle $\mathbb S_\rho =\{ \rho e^{i \theta} \colon 0 \le \theta \le 2 \pi \}$, we have
\[0 \le \frac{\partial \Theta (z)}{\partial \theta} = 2 |z|^2 \im \frac{\Theta_{\bar z}}{z}\]
On the other hand
\[0 \le \re \frac{f(z)}{z}  = \im \frac{i f(z)}{z} = \pm \frac{|f|}{|\Theta_z|} \im \frac{\Theta_{\bar z}}{z} \]
Thus $\lambda (z)>0$, almost everywhere.
From (\ref{Eq172}) we deduce that the real valued function $\lambda$, called the integrating factor, is uniformly bounded from below and from above
\begin{equation}\label{Eq175}
  m \big(|z|\big) \leqslant  \lambda (z) \le    M  \big(|z|\big)
\end{equation}
where $ m(t)$ and $ M (t)$ are positive continuous function in $\mathbb R_+$. Equation (\ref{Eq174}) simply means that $f$ is orthogonal to the gradient field $\nabla \Theta$.
\section{Remarks}
\begin{remark}\label{rem111}
As one moves along a trajectory of a field $f\in \mathcal F_K(d)$ in the positive direction its distance $|x(t)|$ to the critical point  strictly increases to $\infty$. It is a simple matter to see that the length of the tangent vector $\left| \dot{x} (t)\right|$ also increases. Indeed, for sufficiently small $h$ we can write
\[\frac{\dtext }{\dtext t} \left| x(t+h)-x(t) \right|^2 = 2 \langle \dot{x}(t+h)- \dot{x}(t), x(t+h)-x(t)  \rangle \ge 0\]
Hence
\[|x(t+h)-x(t)| \ge |x(s+h)-x(s)| \, , \quad \textnormal{for } t \ge 0\]
We divide by $h$ and let it go to zero to conclude that
\[\left| \dot{x}(t) \right| \ge \left| \dot{x}(s) \right|\]
or, equivalently
\[\left| f(x(t)) \right| \ge \left| f(x(s)) \right| \]
\end{remark}
\begin{remark}\label{rem112}
It is interesting to note that the orthogonal trajectories, locally defined by the autonomous system
\begin{equation}
\dot{w}(t)=if\big( w(t)\big)\, , \quad \quad \alpha < t < \beta
\end{equation}
have well defined curvature at almost every $t\in (\alpha, \beta)$, and it is nonnegative. To carry out the details of this observation we call on the classical formula for the curvature of $C^2$-simple arc $w=w(t)$
\begin{equation}
k=\im \frac{\ddot{w} \, \overline{\dot{w}}}{  \left| \dot{w} \right|^3} = \frac{1}{ \left| \dot{w} \right|}\im \frac{\ddot{w}}{\dot{w}}=  \frac{1}{ \left| \dot{w} \right|} \frac{\dtext}{\dtext t} \left[ \arg \dot{w} \right]
\end{equation}
On the other hand, when $w=w(t)$ is the orthogonal trajectory of $f$,  we may perform the following computation
\[\re \frac{f \big(w(t+h) \big) - f \big(w(t) \big)}{ w(t+h)  -w(t)} \ge 0\]
by monotonicity of $f$. This translates into the inequality
\[  \im \frac{\dot{w}(t+h) - \dot{w}(t)}{w(t+h)-w(t)} \ge 0\]
or, equivalently
\[  \frac{\dtext}{\dtext t} \arg \left[ w(t+h)-w(t) \right]  \ge 0\]
This means that the function $t \to \arg [w(t+h)-w(t)]$ is increasing, so we may write
\[\arg \frac{w(t+h)-w(t)}{h} \ge \arg \frac{w(s+h)-w(s)}{h} \]
for $t \ge s$ and $h$ sufficiently small. Letting $h$ go to zero we arrive at the inequality
\[\arg \dot{z}(t) \ge \arg \dot{z}(s)\, , \quad \quad t \ge s\]
This shows that the function $t \to \arg \dot{z}(t)$ is increasing and, as such, has a nonnegative derivative at almost every point $t\in (\alpha, \beta)$. Now we can define the curvature by the rule
\begin{equation}
k = \frac{1}{\left|\dot{w}(t)\right|} \frac{\dtext}{\dtext t} \left[\arg \dot{w}(t)\right] \ge 0
\end{equation}
or
\begin{equation}
k= \frac{1}{|f(w)|} \frac{\dtext}{\dtext t} \left[\arg f(w)\right]
\end{equation}
\end{remark}

\section{Variation of $\delta$-monotone mappings along $C^1$-arcs, proof of Theorem \ref{Th41}}\label{SecDMR}
\begin{proof}
Let $\Gamma= \{x(t) \colon -\epsilon <t < \epsilon \}$, where $x\colon  (-\epsilon , \epsilon ) \to \Omega$ is a $C^1$-parametrization of $\Gamma$. Here we assume that $\dot{x}(0) \ne 0$. In particular, $x(t) \ne x(s)$ whenever the parameters $t \ne s$ are close to $0$; that is, we assume that $\epsilon$ is sufficiently small.   We shall construct a strictly increasing homeomorphism $\tau \colon (-\eta, \eta )
\overset{\textnormal{\tiny{into}}}{\longrightarrow} (-\epsilon , \epsilon)$, $\tau (0)=0$, such that the function $s \to f \big(x(\tau (s))\big)$ becomes Lipschitz continuous. Obviously, this is enough to claim that $f(\Gamma)$ is rectifiable near the given point $f\big(x(0)\big)$. By means of translation of $\Omega $ and $f(\Omega)$ we may (and do) assume that $x(0)=0$ and $f(0)=0$. Consider two parameters in $(-\epsilon , \epsilon)$, say $-\epsilon < s <t < \epsilon$. In view of $\delta$-monotonicity, we have
\begin{equation}\label{Eq41}
\left\langle f \big(x(t)\big) - f \big(x(s)\big) , \frac{x(t)-x(s)}{|x(t)-x(s)|} \right\rangle \geqslant \delta \left|  f \big(x(t)\big) - f \big(x(s)\big)  \right|
\end{equation}
On the other hand, since $x\in C^1(-\epsilon , \epsilon)$ we also have
\begin{equation}
\lim\limits_{t,s \to 0} \left| \frac{x(t)-x(s)}{|x(t)-x(s)|}  - \frac{\dot{x}(0)}{|\dot{x}(0)|}  \right|=0
\end{equation}
In particular, we find an interval $(-\epsilon^\prime, \epsilon^\prime) \subset (-\epsilon , \epsilon)$ such that
\begin{equation}
 \left| \frac{x(t)-x(s)}{|x(t)-x(s)|}  - \frac{\dot{x}(0)}{|\dot{x}(0)|}  \right| \leqslant \frac{\delta }{2}
\end{equation}
for all  $s,t \in (-\epsilon^\prime, \epsilon^\prime)$. This condition together  with (\ref{Eq41}) yields
\begin{equation}\label{Eq42}
\left\langle f \big(x(t)\big) - f \big(x(s)\big) , \frac{\dot{x}(0)}{|\dot{x}(0)|} \right\rangle \geqslant \frac{\delta}{2} \left|  f \big(x(t)\big) - f \big(x(s)\big)  \right| >0
\end{equation}
It shows that the function $\varphi(t) = \left\langle f \big(x(t)\big) , \frac{\dot{x}(0)}{|\dot{x}(0)|} \right\rangle $ is strictly increasing for $\epsilon^\prime <t < \epsilon^\prime$. It vanishes at $t=0$. The image of the interval $ (-\epsilon^\prime, \epsilon^\prime)$ under $\varphi$ covers an interval $(-\eta, \eta)$. Let $\tau= \tau(s)$ denote the inverse of $\varphi$, defined for $-\eta < s < \eta$. By the definition
\begin{equation}
\left\langle f \big(x(\tau (s))\big) , \frac{\dot{x}(0)}{|\dot{x}(0)|} \right\rangle =s \, , \quad \quad -\eta < s < \eta
\end{equation}
It is this  function $\tau (s)$ that gives us a Lipschitz parametrization of $f(\Gamma)$. We set
\begin{equation}
y(s)=f \big(x(\tau (s))\big) \in f(\Gamma) \, , \quad \quad -\eta < s < \eta
\end{equation}
Then $y$ satisfies a Lipschitz condition with constant $\frac{2}{\delta}$. Indeed,  for $s_1>s_2$, in view of (\ref{Eq42}), we have
\begin{eqnarray}
\left|y(s_1) -y(s_2)\right| &=&  \left|f \big(x(\tau (s_1))\big) - f \big(x(\tau (s_2))\big) \right| \nonumber \\
&\leqslant & \frac{2}{\delta} \left\langle f \big(x(\tau (s_1))\big) - f \big(x(\tau (s_2))\big) , \frac{\dot{x}(0)}{|\dot{x}(0)|} \right\rangle \nonumber \\
&=& \frac{2}{\delta} (s_1-s_2) \nonumber
\end{eqnarray}
as derired.
\end{proof}

\section{Quasiconformal fields  of bounded variation on  $C^1$-curves, proof of Theorem \ref{ThMain4}}\label{SecUniC1toRe}
\begin{proof}
It suffices to prove forward uniqueness. The backward uniqueness follows by considering the field $-f$. We choose $\epsilon >0$ small enough to satisfy
\begin{equation}\label{Eq201}
\left|f\big(x(t)\big)-f(a) \right| + \left|f\big(y(t)\big)-f(a) \right| \le \frac{1}{2} \min \{1, |f(a)|\}
\end{equation}
for all $0 \le t \le \epsilon$.
Suppose, to the contrary, that for  some $0< t_0 < \epsilon$ we have $x(t_0) \ne y(t_0)$. Consider the sequence $t_0 > t_1 > \cdots t_k >t_{k+1} \cdots t_\infty \ge 0$ constructed in Partition Lemma \ref{ApLe1}. Accordingly, we have
\begin{equation}\label{Eq203}
\underset{t_{k+1} \le t,s \le t_k}{\inf} |x(t)-x(s)| =t_k -t_{k+1}\, , \quad x(t_\infty)= y(t_\infty)
\end{equation}
Let us denote by $x_k =x(t_k)$ and $y_k =y(t_k)$. For each $k$ we compute
\begin{eqnarray}\label{Eq204}
\log \frac{|x_k-y_k|}{|x_{k+1}-y_{k+1}|}&=& \int_{t_{k+1}}^{t_k} \frac{\dtext }{\dtext t} \log |x(t)-y(t)|\, \dtext t\nonumber \\
&\le & \int_{t_{k+1}}^{t_{k}} \frac{ \left| \dot{x} (t) - \dot{y}(t) \right|   }{\left|  {x} (t) -  {y}(t) \right|} \, \dtext t = \int_{t_{k+1}}^{t_{k}} \frac{ \left| f \big( x(t) \big) - f \big( y(t) \big) \right|   }{\left|  {x} (t) -  {y}(t) \right|} \, \dtext t \nonumber \\
&\le& \frac{1}{t_k -t_{k+1}} \int_{t_{k+1}}^{t_k} \left| f\big(x(t)\big)- f\big(y(t)\big)  \right| \, \dtext t
\end{eqnarray}
We claim that
\begin{equation}\label{Eq205}
\left| f\big(x(t)\big)- f\big(y(t)\big) \right| \le C \left| f(x_k)-f(x_{k+1}) \right|\, , \quad t_{k+1} \le t \le t_k
\end{equation}
with a constant independent of $k$. To this end, observe that the expressions
 $$|\dot{x}(t)|=|f(x(t))| \quad \mbox{and} \quad  |\dot{y}(t)|=|f(y(t))|$$ are bounded by $\frac{3}{2}|f(a)|$. Hence
\begin{eqnarray*}
\underset{t_{k+1} \le t,s \le t_k}{\sup} |x(t)-y(s)| &\le & \underset{t_{k+1} \le t,s \le t_k}{\inf} |x(t)-y(s)| + 3|f(a)||t_k-t_{k+1}| \\
&=& \left( 1+3|f(a)| \right) (t_k -t_{k+1})
\end{eqnarray*}
Also $|x_k-y(t)| \ge t_k -t_{k+1}$, by (\ref{Eq203}).

Next we see that
\[x_{k+1}-x_k = \int_{t_{k+1}}^{t_k} \left[ f \big( x(\tau)\big) -f(a) \right]\, \dtext \tau + (t_{k}-t_{k+1}) f(a)  \]
Hence
\begin{eqnarray*}
\left|   x_{k+1}-x_k \right| &\ge & (t_{k}-t_{k+1}) |f(a)| - \frac{1}{2} |f(a)| (t_k-t_{k+1}) \\ &=& \frac{1}{2} (t_{k}-t_{k+1}) |f(a)|
\end{eqnarray*}
Now, the three point condition yields
\begin{eqnarray*}
\frac{ \left| f \big( x(t) \big) - f \big( y(t) \big) \right| }{  \left| f (x_k) - f (x_{k+1}) \right| } &\le& \frac{ \left| f \big( x(t) \big) - f \big( y(t) \big) \right| }{  \left| f (x_k) - f \big( y(t) \big)\right| } \cdot \frac{ \left| f (x_k)- f \big( y(t) \big) \right| }{  \left| f (x_k) - f (x_{k+1}) \right| }  \nonumber \\
&\le& M_K \left( \frac{|x(t)-y(t)|}{|x_k-y(t)|} \right) \cdot  M_K \left( \frac{|x_k-y(t)|}{|x_k-x_{k+1})|} \right) \nonumber \\
&\le & M_K \left(  {1+3|f(a)|}  \right) \cdot  M_K \left( \frac{2}{|f(a)|} +6 \right) \nonumber \\
&=& C
\end{eqnarray*}
This proves (\ref{Eq205}).

We now substitute (\ref{Eq205}) into (\ref{Eq204}) to obtain
\begin{equation}
\log \frac{|x_k-y_k|}{|x_{k+1}-y_{k+1}|} \le C \, \left| f(x_{k+1}) -f(x_k) \right|
\end{equation}
Using telescoping structure on the left hand side we compute
\begin{equation}
\log \frac{|x_0-y_0|}{|x_{\ell}-y_{\ell}|} \le C \sum_{k=0}^\ell \left| f(x_{k+1}) -f(x_k) \right|
\end{equation}
Finally, letting $\ell$ go to infinity we see that the left hand side approaches $\infty$, because
\[x_\ell -y_\ell \to x(t_\infty) -y(t_\infty) =0\]
However, the right hand side is bounded by the total variation of $f$ along the $C^1$-curve $x=x(t)$. This contradiction completes the proof of Theorem \ref{ThMain4}.
\end{proof}

\section{Examples}
\begin{example}\label{Ex1}
Consider the complex function
\begin{equation}
f(z)= \frac{10 z}{\sqrt{|z|}} \, , \quad z=x_1+ix_2
\end{equation}
It satisfies the   reduced Beltrami equation
\begin{equation}
f_{\bar z} = \mu (z) \re f_z\, , \quad \mu (z)= -   \frac{1}{3}\,  \frac{z}{\bar z}
\end{equation}
and is $\delta$-monotone. Nevertheless, there are two integral curves passing through the origin
\begin{equation}
z^\pm (t)= \begin{cases} (24 \pm 7i)\,  t^2 & \quad  \textnormal{if} \quad   t \geqslant 0 \\
0 & \quad \textnormal{if} \quad t \leqslant 0  \end{cases}
\end{equation}
Note that $f$ is H\"older continuous with exponent $\alpha = \frac{1}{2}$.
\end{example}

\begin{example}\label{Ex111}
It is not difficult to construct a (nondegenerate) reduced $K$-quasiconformal field $f$ for which any factorization of the form
\[i\, f(z)= \lambda (z) \nabla U(z)\, , \quad \lambda (z)\in \mathbb R \quad U\mbox{-locally Lipschitz in } C_0 \]
does not allow $\lambda$ to be continuous, equivalently $U$ to be $C^1$-smooth. Set
\begin{equation}
f(z)= \begin{cases} 2z & \im z \geqslant 0 \\
3z - \bar z \; \; & \im z \leqslant 0 \end{cases}
\end{equation}
Indeed, $U$ must be constant on every integral curve of $f$, among which are half lines
\[ y=cx \quad c \ge 0 \quad x \ge 0\]
and
half-parabolas
\[ y=cx^2 \quad c \le 0 \quad x \ge 0\]
This forces $U$ to be of the form $U(x,y)=A(y/x)$ in the first quadrant and  $U(x,y)=B(y/x^2)$ in the fourth quadrant.
It follows that $\lim\limits_{y\to 0+} U_y (x,y) = a/x$ and $\lim\limits_{y\to 0-} U_y (x,y) = b/x^2$ where $a$ and $b$ are nonzero constants because $U_x(x,0) \equiv 0$. This contradicts the smoothness of $U$.
\end{example}

\bibliographystyle{amsplain}

\end{document}